\documentclass[12pt]{article}

\usepackage[colorlinks=true, pdfstartview=FitV, linkcolor=blue, citecolor=blue, urlcolor=blue]{hyperref}

\usepackage{amssymb,amsmath, amscd, mathdots}
\usepackage{times, verbatim}
\usepackage{graphicx}
\usepackage[all]{xy}

\DeclareFontFamily{OT1}{rsfs}{}
\DeclareFontShape{OT1}{rsfs}{n}{it}{<-> rsfs10}{}
\DeclareMathAlphabet{\mathscr}{OT1}{rsfs}{n}{it}

\newtheorem{theorem}{Theorem}
\newtheorem{corollary}[theorem]{Corollary}
\newtheorem{lemma}[theorem]{Lemma}
\newtheorem{proposition}[theorem]{Proposition}
\newenvironment{proof}{\noindent {\bf Proof:}}{\hfill$\Box$ \vspace{2 ex}}

\DeclareMathOperator{\End}{End}

\DeclareMathOperator{\Gal}{Gal}

\DeclareMathOperator{\SO}{SO}

\DeclareMathOperator{\Spin}{Spin}

\DeclareMathOperator{\SL}{SL}
\DeclareMathOperator{\GL}{GL}

\DeclareMathOperator{\Tr}{Trace}
\DeclareMathOperator{\Trace}{Trace}

\DeclareMathOperator{\disc}{disc}
\DeclareMathOperator{\Res}{Res}
\DeclareMathOperator{\Spec}{Spec}

\DeclareMathOperator{\Sym}{Sym}
\DeclareMathOperator{\Span}{Span}

\def\disc{{\rm disc}}
\def\dim{{\rm dim}}

\def\Q{{\mathbb Q}}

\def\Z{{\mathbb Z}}
\def\P{{\mathbb P}}

\def\Q{{\mathbb Q}}

\def\Res{{\textrm{Res}}}

\def\Rlk {\mathrm{Res}_{L/k}}

\def\TIMES {\times}
\def\reductive{reductive}
\newcommand{\gl}[2]{{^{#2\!}#1}}
\newcommand{\w}[1]{\widetilde{#1}}

\title{Arithmetic invariant theory II}

\author{Manjul Bhargava, Benedict H.\ Gross, and Xiaoheng Wang}

\begin{document}
\maketitle

\tableofcontents

\section{Introduction}

Geometric invariant theory involves the study of invariant polynomials for the action of a reductive algebraic group $G$ on a linear representation $V$ over a field $k$, and the relation between these invariants and the $G$-orbits on $V$, usually under the hypothesis that the base field $k$ is algebraically closed. In favorable cases, one can determine the geometric quotient $V/\!\!/G = \Spec(\Sym^*(V^\vee))^G$ and identify certain fibers of the morphism $V \rightarrow V/\!\!/G$ with certain $G$-orbits on $V$.
For general fields $k$ the situation is more complicated. 
The additional complexity in the orbit picture, when $k$ is not separably closed, is what we refer to as {\it arithmetic invariant theory}.

In a previous paper \cite{BG}, we studied the arithmetic invariant theory of a reductive group $G$ acting on a linear representation $V$ over a general field $k$. Let $k^s$ denote
a separable closure of $k$. When the
stabilizer $G_v$ of a vector $v$ is smooth, the $k$-orbits inside of the $k^s$-orbit
of $v$ are parametrized by classes in the kernel of the map of pointed sets in Galois cohomology
$\gamma: H^1(k,G_v) \rightarrow H^1(k,G)$ (cf. \cite{S}).

We produced elements in the kernel of $\gamma$ for three representations of the split odd orthogonal
group $G = \SO(W) = \SO(2n+1)$: the standard representation $V = W$, the adjoint
representation $V = \wedge^2(W)$, and the symmetric square representation
$V = \Sym^2W$. For all three representations the ring of $G$-invariant polynomials on $V$ is a polynomial ring
and the categorical quotient $V/\!\!/G$ is isomorphic to affine space. Furthermore, in each case there is a natural section
of the morphism $\pi: V \rightarrow V/\!\!/G$, so the $k$-rational points of $V/\!\!/G$
lift to $k$-rational orbits of $G$ on $V$.

Such a section may not exist for the action of the odd orthogonal groups $G' = \SO(W')$ that are not split
over $k$. The corresponding representations $V' = W'$, $\wedge^2(W')$, and $\Sym^2W'$ have the same
ring of polynomial invariants, so $V'/\!\!/G' = V/\!\!/G$, but there may be rational points in this affine space
that do not lift to rational orbits of $G'$ on $V'$. 

The groups $G' = \SO(W')$ are the
{\it pure inner forms} of $G$, i.e., twists of $G$ using classes in $H^1(k, G)$ as opposed to inner forms of $G$ which are obtained using classes in $H^1(k, G^{\rm ad})$.  The representations $V'$ of $G'$ are then the corresponding twists of $V$. These pure
inner forms and their representations are determined by a class $c$ in the pointed set $H^1(k,G)$. Suppose
that the image of $v$ in $V/\!\!/G$ is equal to $f$, and that $G(k^s)$ acts transitively on the $k^s$-rational
points of the fiber above $f$. Then we show
that the $k$-orbits for $G'$ on $V'$ with invariant $f$
are parametrized by the elements in the
fiber of the map $\gamma: H^1(k,G_v) \rightarrow H^1(k,G)$ above the class $c$.

We also consider representations where there is an obstruction to lifting $k$-rational
invariants in $V/\!\!/G$ to $k$-rational orbits on $V$, for {\bf all} pure inner forms of $G$.
Let $f$ be a rational invariant in $V/\!\!/G$, and assume that there is a single orbit over $k^s$
with invariant $f$, whose stabilizers $G_v$ are {\bf abelian}. We show that these stabilizers
are canonically isomorphic to a fixed commutative
group scheme $G_f$, which is determined by $f$ and is defined over $k$.
We then construct a class
$d_f$ in the cohomology group $H^2(k,G_f)$, whose non-vanishing
obstructs the descent of the orbit to $k$, for all pure inner forms of $G$. On the other hand, if
$d_f = 0$, we show that there is least one pure inner form of $G$ that has $k$-rational orbits
with invariant $f$.

When the stabilizer $G_v$ is trivial, so the action of $G(k^s)$ on elements with invariant $f$ over $k^s$ is simply transitive, the obstruction $d_f$ clearly vanishes. In this case, we show that there is a unique pure inner form $G'$ for which there exists a unique
$k$-rational orbit on $V'$ with invariant $f$.
We give a number of examples of such representations, such as the action
of $\SO(W) = \SO(n+1)$ on $n$ copies of the standard representation $W$, and the action of $\SL(W) = \SL(5)$ on
three copies of the exterior square representation $\wedge^2(W)$.

It is also possible that the stabilizer $G_v$ is abelian and nontrivial, and yet the obstruction $d_f$~still vanishes.  This scenario occurs frequently; for example, it occurs for all representations arising in Vinberg's theory of $\theta$-groups~(see \cite{PV} and \cite{P}).  These representations are remarkable in that the morphism $\pi:V\rightarrow V/\!\!/G$ has an (algebraic) section (called the {\it Kostant section}). This implies that the obstruction $d_f$ vanishes.  The representations
$\wedge^2(W)$ and $\Sym^2W$ of the odd split orthogonal group $\SO(W)$ studied
in \cite{BG} indeed shared this property.
(For a treatment of many such representations of arithmetic interest, involving rational points and Selmer groups of Jacobians of algebraic curves, see \cite{BH}, \cite{BG1}, \cite{G}, \cite{SW}, and \cite{T}.)

Finally, it is possible that the stabilizer $G_v$ of a stable vector $v$ is abelian and nontrivial, and the obstruction class $d_f$ is also nontrivial in $H^2(k,G_v)$.  Fewer such representations occur in the literature, but they too appear to be extremely rich arithmetically.  In this paper, we give a detailed study of such a representation, namely
the action of $G = \SL(W) = \SL(n)$ on the vector space $V = \Sym_2W^* \oplus \Sym_2W^*$ of
pairs of symmetric bilinear forms on $W$. Like the representation $\Sym^2W$ of $\SO(W)$, the
ring of polynomial invariants is a polynomial ring, and there
are stable orbits in the sense of geometric invariant theory. In fact, the stabilizer $G_v$ of any vector $v$ in one of the stable orbits is a finite
commutative group scheme isomorphic to $(\Z/2\Z)^{n-1}$ over $k^s$, and $G(k^s)$ acts transitively on the vectors in $V(k^s)$ with the
same invariant $f$ as $v$.  However, when the dimension $n = 2g+2$ of $W$ is even, it may not be possible to
lift $k$-rational points $f$ of the quotient $V/\!\!/G$ to $k$-rational orbits of $G$ on
$V$. We relate this obstruction to the arithmetic
of $2$-coverings of Jacobians of hyperellipic curves of genus $g$ over $k$.

In~\cite{genhyper}, this connection with hyperelliptic curves was used to show that most hyperelliptic curves over $\Q$ of genus $g\geq 2$ have no rational points.  In a forthcoming paper \cite{BGW}, we will use the full connection with 2-coverings of Jacobians of hyperelliptic curves to study the arithmetic of hyperelliptic curves; in~particular, we will prove that a positive proportion of hyperelliptic curves over~$\Q$ have points locally over~$\Q_\nu$ for all places $\nu$ of $\Q$, but have no points globally over {\it any} odd degree extension of $\Q$.

This paper is organized as follows.  In Section~\ref{lifting}, we describe the notion of a pure inner form $G'$ of a \reductive group $G$ over a field $k$, and the corresponding twisted form $V'$ of a given representation $V$ of $G$.  We also discuss in detail the problem of lifting $k$-rational points of $V/\!\!/G$ to $k$-rational orbits of $G$ (and its pure inner forms) in the case where the  generic stabilizer $G_v$ is abelian, and we describe the cohomological obstruction to lifting invariants lying in $H^2(k,G_f)$.  
The obstruction element in $H^2(k,G_f)$ can also be deduced from the theory of residual gerbes on algebraic stacks (see~\cite{Giraud} and \cite[Chapter~11]{LMB}).  Since we have not seen any concise reference to the specific results needed in this context, we felt it would be useful to give a self-contained account here.

In Section~\ref{trivstab}, we then consider three examples of representations where the stabilizer $G_v$ is trivial.  These representations are:
\begin{enumerate}
\item the split orthogonal group $\SO(W)$ acting on $n$ copies of $W$, where $\dim(W)=n+1$;
\item $\SL(W)$ acting on three copies of $\wedge^2W$, where $\dim(W)=5$;
\item the unitary group $U(n)$ acting on the adjoint representation of $U(n+1)$.
\end{enumerate}
In each of these three cases, the cohomological obstruction clearly vanishes and we see explicitly how the orbits, over all pure inner forms of the group $G$, are classified by the elements of the space $V/\!\!/G$ of invariants.  The third representation and its orbits have played an important role in the work of Jacquet--Rallis~\cite{JR} and Wei Zhang \cite{Z} in connection with the relative trace formula approach to the conjectures of Gan, Gross, and Prasad~\cite{GGP}.

In Section~\ref{nontrivstab}, we study three examples of representations where the stabilizer $G_v$ is nontrivial and abelian, and where there are cohomological obstructions to lifting invariants. These representations are:
\begin{enumerate}
\item $\Spin(W)$ acting on $n$ copies of $W$, where $\dim(W)=n+1$;
\item $\SL(W)$ acting on $\Sym_2W^* \oplus \Sym_2W^*$;
\item $(\SL/\mu_2)(W)$ acting on $\Sym_2W^* \oplus \Sym_2W^*$ (this group acts only when $\dim(W)$ is even).
\end{enumerate}
In the first case, we show that the obstruction is the Brauer class of a Clifford algebra determined by the invariants.
In the second and third cases, we show that when $n$ is odd, there is no cohomological obstruction to lifting invariants, but when $n$ is even, the obstruction can be nontrivial.  We parametrize the orbits for both groups in terms of arithmetic data over $k$, and describe the resulting criterion for the existence of orbits over~$k$.  Finally, we describe the connection between the cohomological obstruction and the arithmetic of two-covers of Jacobians and hyperelliptic curves over $k$, which will play an important role in \cite{BGW}.

As in \cite{BG}, the heart of this paper lies in the examples that illustrate the various scenarios
that can occur, and how one can treat each scenario in order to classify the orbits, over a field that is not necessarily algebraically closed, in terms of suitable arithmetic data.

We thank Jean-Louis Colliot-Th\'el\`ene, Jean-Pierre Serre,
Bas Edixhoven, and Wei Ho
for useful conversations and for their help with the literature.

\section{Lifting results}\label{lifting}

In this section, we assume that $G$ is a reductive group with a linear representation $V$ over the field $k$. We will study the general problem of lifting $k$-rational points of $V/\!\!/G$ to $k$-rational orbits of pure inner forms $G'$ of $G$ on the corresponding twists $V'$ of $V$. For stable orbits over the separable closure $k^s$ with smooth abelian stabilizers $G_v$, we will show how these stabilizers descend to a group scheme $G_f$ over $k$ and describe a cohomological obstruction to the lifting problem lying in $H^2(k,G_f)$.

\subsection{Pure inner forms}

 We begin by recalling the notion of a pure inner form $G^c$ of $G$ and the action of $G^c$ on a twisted representation $V^c$ (\cite[Ch 1 \S5]{S}).

Suppose $(\sigma \rightarrow c_\sigma)$ is a $1$-cocycle on $\text{Gal}(k^s/k)$ with values in the group $G(k^s)$. That is, $c_{\sigma\tau} = c_\sigma\cdot\gl{c_\tau}{\sigma}$ for any $\sigma,\tau\in\text{Gal}(k^s/k)$.
We define the pure inner form $G^c$ of $G$ over $k$ by giving its $k^s$-points and describing a Galois action. Let $G^c(k^s) = G(k^s)$ with action
\begin{equation}\label{eq:twistact}
\sigma(h) = c_{\sigma}\gl{h}{\sigma}c_\sigma^{-1}.
\end{equation}
Since $c$ is a cocycle, we have $\sigma\tau(h) = \sigma(\tau(h))$.

Let $g$ be an element of $G(k^s)$. If $b_\sigma = g^{-1}c_\sigma \gl{g}{\sigma}$ is a cocycle in the same cohomology class as~$c$, then the map
on $k^s$-points $G^b \rightarrow G^c$ defined by $h \to ghg^{-1}$ commutes with the respective Galois actions, so defines an isomorphism over $k$. Hence the isomorphism
class of the pure inner form $G^c$ over $k$ depends only on the image of $c$ in the pointed set $H^1(k,G)$.

\subsection{Twisting the representation}

If we compose the cocycle $c$ with values in $G(k^s)$ with the homomorphism $\rho: G \rightarrow \GL(V)$, we obtain a cocycle $\rho(c)$ with values in $\GL(V)(k^s)$. By the generalization of Hilbert's Theorem~$90$, we have $H^1(k,\GL(V)) = 1$ (\cite[Ch X]{S2}). Hence there is an element $g$ in $\GL(V)(k^s)$, well-defined up to left multiplication by $\GL(V)(k)$, such that
\begin{equation}\label{trivialization}\rho(c_\sigma) = g^{-1}\gl{g}{\sigma}\end{equation}
for all $\sigma$ in $\Gal(k^s/k)$.

We use the element $g$ to define a twisted representation of the group $G^c$ on the vector space $V$ over $k$. The homomorphism
$$\rho_g: G^c(k^s) \rightarrow \GL(V)(k^s)$$
defined by $\rho_g(h) = g\rho(h)g^{-1}$ commutes with the respective Galois actions, so defines a representation over $k$. We emphasize that the Galois action on $G^c(k^s)$ is as defined in \eqref{eq:twistact}, whereas the Galois action on $\GL(V)(k^s)$ is the usual action.

The isomorphism class of the representation $\rho_g: G^c \rightarrow \GL(V)$ over $k$ is independent of the choice of $g$ in (\ref{trivialization}) which trivializes the cocycle. If $g' = ag$ is another choice, with $a$ in $\GL(V)(k)$, then conjugation by $a$ gives an isomorphism from $\rho_g$ to $\rho_g'$. Since the isomorphism class of this representation depends only on the cocycle $c$, we will write $V^c$ for the representation $\rho_g$ of $G^c$.

The fact that the cocycle $c_\sigma$ takes values in $G$, and not in the adjoint group, is crucial to defining the twist $V^c$ of the representation $V$. For 1-cocycles $c$ with values in $G^{\text{ad}}\hookrightarrow \text{Aut}(G)$, one can define the inner form $G^c$, but one does not always obtain a twisted representation $V^c$. For example, consider the case of $G=\SL_2$ with $V$ the standard two-dimensional representation. The nontrivial inner forms of $G$ are obtained from nontrivial cohomology classes in $H^1(k, \text{PGL}_2)$. These are the groups $G^c$ of invertible elements of norm 1 in quaternion division algebras $D$ over $k$. The group $G^c$ does not have a faithful two-dimensional representation over $k$---this representation is obstructed by the quaternion algebra $D$. Since $H^1(k,\SL_2)$ is trivial, there are no nontrivial \textit{pure} inner forms of $G$.

\subsection{Rational orbits in the twisted representation}

We now fix a rational point $f$ in the canonical quotient  $V/\!\!/G$, and let $V_f$ be the fiber in $V$. For the rest of this subsection, we assume that the set $V_f(k)$ of rational points in the fiber is nonempty, 
and that $G(k^s)$ acts transitively on the points in $V_f(k^s)$. 
In particular, this orbit is closed (as it is defined by the values of the invariant polynomials).  Let $v$ be a point in $V_f(k)$ and let $G_v$ denote its stabilizer in~$G$.

The group $G(k)$ acts on the rational points of the fiber over $f$. In Proposition $1$ of \cite{BG} we showed that the orbits of $G(k)$ on the set $V_f(k)$ correspond bijectively to elements in the kernel of the map 
$$\gamma: H^1(k,G_v) \rightarrow H^1(k,G)$$
 of pointed sets in Galois cohomology.
In this section, we will generalize this to a parametrization of certain orbits of $G^c(k)$, where $c\in H^1(k,G)$.
Note that by our hypothesis and the definition of $G^c$, the group $G^c(k^s) = G(k^s)$ acts transitively on the set $gV_f(k^s)$ in $V(k^s)$, where $g$ is as in (\ref{trivialization}). We define the~set
$$V_f^c(k) := V(k) \cap gV_f(k^s),$$
which admits an action of the rational points of the pure inner form $G^c$.

Here is a simple example, which illustrates
many elements of the theory of orbits for pure inner twists with a fixed rational invariant $f$. Assume that
the characteristic of $k$ is not equal to $2$, and let $G$ be the \'etale group scheme $\mu_2$ of
order $2$ over $k$. Let $V$ be the nontrivial one-dimensional representation of $G$ on the field $k$.
(This is the standard representation of the orthogonal group $O(1)$ over $k$.) The polynomial invariants of this
representation are generated by $q(x) = x^2$, so the canonical quotient $V/\!\!/G$ is the affine line.
Let $f$ be a rational invariant in $k$ with $f \neq 0$. Then the fiber $V_f$ is the subscheme of $V$
defined by $\{x: x^2 = f\}$, so $V_f(k)$ is nonempty if and only if $f$ is a square in $k^\times$. This is
certainly true over the separable closure $k^s$ of $k$, and the group $G(k^s)$ acts simply transitively
on $V_f(k^s)$.

An element $c$
in $k^\times$ defines a cocycle $c_{\sigma} = \gl{\sqrt c}{\sigma}/\sqrt c$ with values in $G(k^s)$, whose class
in the cohomology group $H^1(k,G) = k^\times/k^{\times2}$ depends only on the image
of $c$ modulo squares. The element $g = \sqrt c$ ~ in $\GL(V)(k^s)$ trivializes this class in the
group $H^1(k,\GL(V))$. Although the inner twist $G^c$ and the representation $V^c$ remain exactly
the same, we find that
$$V_f^c(k) = V(k) \cap gV_f(k^s) = \{x \in k^\times: x^2 = fc\}.$$
Hence the set $V_f^c(k)$ is nonempty if and only if the element $fc$ is a square in $k^\times$. Note that there is
a unique inner twisting $G^c$ where the fiber $V_f^c$ has $k$-rational points, and in that case the group $G^c(k)$
acts simply transitively on $V_f^c(k)$.

Returning to the general case, we have the following generalization of Proposition $1$ in $\cite{BG}$ (which is the case $c = 1$ below).

\begin{proposition}\label{prop:twistorbit}
Let $G$ be a reductive group with representation $V$. Suppose there exists $v\in V(k)$ with invariant $f\in (V/\!\!/G)(k)$ and stabilizer $G_v$ such that $G(k^s)$ acts transitively on $V_f(k^s)$. Then there is a bijection between the set of $G^c(k)$-orbits on $V_f^c(k)$ and the fiber $\gamma^{-1}(c)$ of the map $$\gamma: H^1(k,G_v) \rightarrow H^1(k,G)$$ above the class $c\in H^1(k,G)$.
In particular, the image of $H^1(k, G_v)$ in $H^1(k, G)$ determines the set of pure inner forms of $G$ for which the $k$-rational invariant $f$ lifts to a $k$-rational orbit of $G^c$ on $V^c$.
\end{proposition}

Before giving the proof, we illustrate this with an example from \cite{BG}. Let $W$ be a split orthogonal space of dimension $2n+1$ and signature $(n+1,n)$ over $k = \mathbb R$, let $G = \SO(W) = \SO(n+1,n)$. The pure inner forms of $G$ are the groups $G^c = \SO(p,q)$ with $p+q = 2n+1$ and $q \equiv n$  (mod $2$), and the representation $W^c$ of $G^c$ is the standard representation on the corresponding orthogonal space $W(p,q)$ of signature $(p,q)$. The group $G = \SO(W)$ acts faithfully on the space $V = \Sym_2(W)$ of self-adjoint operators $T$ on $W$. For this representation, the inner twists $G^c$ of $G$ are exactly the same, and the twisted representation $V^c$ of $G^c$ is isomorphic to $\Sym^2W^c$.  The polynomial invariants $f$ in $(V/\!\!/G)(\mathbb R)$ are given by the coefficients of the characteristic polynomial of $T$. Assume that this characteristic polynomial is separable, with $2m+1$ real roots.
Then the stabilizer of a point $v_0 \in V_f(\mathbb R)$ is the finite commutative group scheme $(\mu_2^{2m+1} \times (\Res_{\mathbb C/\mathbb R}\mu_2)^{n-m})_{N = 1}$. Hence $H^1(\mathbb R, G_{v_0})$ is an elementary abelian $2$-group of order $2^{2m}$. This group maps under $\gamma$ to the pointed set $H^1(\mathbb R,\SO(W))$, which is finite of cardinality $n+1$.
The fiber over the class of $\SO(p,q)$ is nonempty if and only if both $p$ and $q$ are greater than or equal to $n-m$. In this case, write $q = n - m + a$, with $a \equiv m$ (mod $2$). Then the fiber has cardinality $\binom{2m+1}{a}$. For example, the kernel has cardinality $\binom{2m+1}{m}$. When $pq = 0$, so the space $W^c = W(p,q)$ is definite, there are orbits in $V^c_f(\mathbb R)$ only in the case when $m = n$, so the characteristic polynomial splits completely over $\mathbb R$. In that case there is a single orbit. This is the content of the classical spectral theorem.

\vspace{.17in}

\noindent \textbf{Proof of Proposition \ref{prop:twistorbit}:} Suppose $c$ is a 1-cocycle with values in $G(k^s)$ and fix $g\in \GL(V)(k^s)$ such that $c_\sigma = g^{-1}\gl{g}{\sigma}$ for all $\sigma\in\text{Gal}(k^s/k).$ When $V^c_f(k)$ is nonempty we must show that $c$ is in the image of $H^1(k, G_v).$ Indeed, suppose $gw\in V^c_f(k)$ for some $w\in V_f(k^s).$ By our assumption on the transitivity of the action on $k^s$ points, there exists $h\in G(k^s)$ such that $w = hv.$ The rationality condition on $gw$ translates into saying that, for any $\sigma\in \text{Gal}(k^s/k)$, we have $c_\sigma \gl{h}{\sigma} v = hv$. That is, $h^{-1}c_\sigma \gl{h}{\sigma}\in G_v$ for any $\sigma\in \text{Gal}(k^s/k)$. In other words, $c$ is in the image of $\gamma$.

Now suppose $c\in H^1(k, G)$ is in the image of $\gamma$. Without loss of generality, assume that $c_\sigma\in G_v(k^s)$ for any $\sigma\in\text{Gal}(k^s/k)$. Pick any $g\in \GL(V)(k^s)$ as in \eqref{trivialization} above and set $w = gv\in V^c_f(k^s).$ Then for any $\sigma\in\text{Gal}(k^s/k)$, we have $$\gl{w}{\sigma} = gc_\sigma v = gv = w.$$ This shows that $w\in V^{c}_f(k).$ Hence there is a bijection between $G^c(k)\backslash V_f^c(k)$ and $\ker\gamma_{c}$ where $\gamma_{c}$ is the natural map of sets $H^1(k, G^{c}_{w})\rightarrow H^1(k, G^{c}).$ To prove Proposition~\ref{prop:twistorbit}, it suffices to establish a bijection between $\gamma^{-1}(c)$ and $\ker\gamma_{c}.$ Consider the following two maps:
\[
\begin{array}{rclcrcl}
\gamma^{-1}(c)&\rightarrow&\ker\gamma_{c}&\qquad&\ker\gamma_{c}&\rightarrow&\gamma^{-1}(c)\\
(\sigma\rightarrow d_\sigma)&\mapsto&(\sigma\rightarrow d_\sigma c_\sigma^{-1})&\qquad&(\sigma\rightarrow a_\sigma )&\mapsto&(\sigma\rightarrow a_\sigma c_\sigma)\\
\end{array}
\]
We need to check that these maps are well-defined. First, suppose $(\sigma\rightarrow d_\sigma)\in \gamma^{-1}(c).$  Then we need to show that $(\sigma\rightarrow d_\sigma c_\sigma^{-1})$ is a 1-cocycle in the kernel of $\gamma_{c}.$ Note that, for any $\sigma,\tau\in \text{Gal}(k^s/k),$ we have
$$(d_\sigma c_\sigma^{-1})\cdot\sigma(d_\tau c_\tau^{-1})\cdot(d_{\sigma\tau} c_{\sigma\tau}^{-1})^{-1}=d_\sigma c_\sigma^{-1}(c_\sigma \gl{d_\tau}{\sigma} \gl{c_\tau^{-1}}{\sigma} c_\sigma^{-1})(d_{\sigma\tau} c_{\sigma\tau}^{-1})^{-1}=1.$$
Moreover, there exists $h\in G(k^s)$ such that $d_\sigma=h^{-1}c_\sigma\gl{h}{\sigma}$ for any $\sigma\in\text{Gal}(k^s/k),$ and thus
$$h^{-1}\sigma(h) = h^{-1}c_\sigma \gl{h}{\sigma}c_\sigma^{-1} = d_\sigma c_\sigma^{-1}.$$
This shows that $(\sigma\rightarrow d_\sigma c_\sigma^{-1})$ is in the kernel of $\gamma_c.$ Likewise, one can show that the second map is also well-defined. The composition of these two maps in either order yields the identity map, and this completes the proof. \hfill$\Box$

\subsection{A cohomological obstruction to lifting invariants}
\label{sec:cohomobst}

Suppose $f\in (V/\!\!/G)(k)$ is a rational invariant. We continue to assume that the group $G(k^s)$ acts transitively on the set $V_f(k^s)$. In this section, we consider the problem of determining when the set $V^c_f(k)$ is nonempty for some $c\in H^1(k,G)$. That is, when does a rational invariant lift to a rational orbit for some pure inner form of $G$? We resolve this problem under the additional assumption that the stabilizer $G_v$ of any point in the orbit $V_f(k^s)$ is abelian.

For $\sigma\in\Gal(k^s/k)$, the vector $\gl{v}{\sigma}$ also lies in $V_f(k^s)$, so there is an element $g_{\sigma}$ with $g_\sigma \gl{v}{\sigma} = v$. The element $g_\sigma$ is well-defined up to left multiplication by an element in the subgroup $G_v$. Since we are assuming that the stabilizers are abelian, the homomorphism
$\theta_\sigma: G_{\gl{v}{\sigma}} \rightarrow G_v$
defined by mapping $\alpha$ to $g_\sigma \alpha g_\sigma^{-1}$ is independent of the choice of $g_\sigma$. This gives a collection of isomorphisms
$$\theta_\sigma: \gl{(G_v)}{\sigma} \rightarrow G_v$$
that satisfy the 1-cocycle condition $\theta_{\sigma \tau} = \theta_\sigma\circ\gl{\theta_\tau}{\sigma}$, and hence provide descent data for the group scheme $G_v$. We let $G_f$ be the corresponding commutative group scheme over $k$ which depends only on the rational invariant $f$. Let $\iota_v:G_f(k^s)\xrightarrow{\sim} G_v$ denote the canonical isomorphisms. More precisely, if $h\in G(k^s)$ and $v\in V_f(k^s)$ then
\begin{equation}\label{eq:canstab}
\iota_{hv}(b) = h\iota_v(b)h^{-1}\quad\forall b\in G_f(k^s).
\end{equation}
The descent data translates into saying for any $\sigma\in \text{Gal}(k^s/k)$ and $v\in V_f(k^s),$ we have
\begin{equation}\label{eq:cangalstab}
\gl{(\iota_v(b))}{\sigma} = \iota_{\gl{v}{\sigma}}(\gl{b}{\sigma})\quad\;\;\;\forall b\in G_f(k^s).
\end{equation}

Before constructing a class in $H^2(k, G_f)$ whose vanishing is intimately related to the
existence of rational orbits, we give an alternate method (shown to us by Brian Conrad) to
obtain the finite group scheme $G_f$ over $k$ using fppf descent. Suppose $G$ is a group scheme of finite type over $k$ such that the orbit map $G\times V_f\rightarrow V_f$ is fppf. Suppose also that the stabilizer $G_v\in G(k^a)$ for any $v\in V_f(k^a)$ is abelian where $k^a$ denotes an algebraic closure of $k$.
Let $H$ denote the stabilizer subscheme of $G \times V_f$. In other words, $H$ is the
pullback of the action map $G\times V_f \rightarrow V_f\times V_f$ over the diagonal of
$V_f$. Note that $H$ is a $V_f$-scheme and its descent to $k$ will be $G_f$. The descent
datum amounts to a canonical isomorphism $p_1^*H \simeq p_2^*H$ where $p_1,p_2$ denote the
two projection maps $V_f\times V_f \rightarrow V_f.$ The commutativity of $G_v$ for any
$v\in V_f(k^a)$ implies the commutativity of $(G_R)_x$ for any
$k$-algebra~$R$ and any element $x \in V_f(R)$. Therefore, there are canonical isomorphisms $(G_R)_x \rightarrow (G_R)_y$
for any $x,y\in V_f(R).$ This gives canonical isomorphisms $p_1^*H \simeq p_2^*H$ locally
over $V_f \times V_f$. Being canonical, these local isomorphisms patch together to a global
isomorphism and hence yield the desired descent datum.

We now construct a class $d_f$ in $H^2(k, G_f)$ that will be trivial whenever a rational orbit exists. Choose $v$ and $g_{\sigma}$ as above, with $g_\sigma \gl{v}{\sigma} = v$. Define
$$d_{\sigma,\tau} = \iota_v^{-1}(g_\sigma\gl{g_\tau}{\sigma} g_{\sigma \tau}^{-1}).$$
Standard arguments show that $d_{\sigma,\tau}$ is a 2-cocycle whose image $d_f$ in $H^2(k,G_f)$ does not depend on the choice of $g_\sigma$. We also check that the $2$-cochain $d_{\sigma,\tau}$ does not depend on the choice of $v\in V_f(k^s)$. Suppose $v' = hv\in V_f(k^s)$ for some $h\in G(k^s).$ For any $\sigma\in \text{Gal}(k^s/k)$, we have
$$h g_\sigma \gl{h^{-1}}{\sigma} \gl{v'}{\sigma} = hg_\sigma\gl{v}{\sigma} = hv = v'.$$
Moreover, for any $\sigma,\tau\in \text{Gal}(k^s/k),$ we compute
$$h g_\sigma \gl{h^{-1}}{\sigma}\,\gl{(h g_\tau \gl{h^{-1}}{\tau})}{\sigma}\,(h g_{\sigma\tau} \gl{h^{-1}}{\sigma\tau})^{-1} = hg_\sigma\gl{g_\tau}{\sigma} g_{\sigma\tau}^{-1}h^{-1};$$
hence, by \eqref{eq:canstab}, we have $$\iota_{v'}^{-1}(h g_\sigma \gl{h^{-1}}{\sigma}\,\gl{(h g_\tau \gl{h^{-1}}{\tau})}{\sigma}\,(h g_{\sigma\tau} \gl{h^{-1}}{\sigma\tau})^{-1}) = \iota_v^{-1}(g_\sigma\gl{g_\tau}{\sigma} g_{\sigma\tau}^{-1}).$$
If $V_f(k)$ is nonempty, then one can take $v$ in $V_f(k)$. Then one can take $g_\sigma=1$ and hence $d_f = 0.$ We have therefore obtained the following necessary condition for lifting invariants to orbits.

\begin{proposition}\label{prop:liftnec}
Suppose that $f$ is a rational invariant, and that $G(k^s)$ acts transitively on $V_f(k^s)$ with abelian stabilizers.
If $V_f(k)$ is nonempty, then $d_f = 0$ in $H^2(k, G_f)$.
\end{proposition}

This necessary condition is not always sufficient. As shown by the following cocycle computation, the class $d_f$ in $H^2(k, G_f)$ does not depend on the pure inner form of $G$. Indeed, suppose $c\in H^1(k, G)$ and $g\in \GL(V)(k^s)$ such that $c_\sigma = g^{-1}\gl{g}{\sigma}$ for all $\sigma\in\text{Gal}(k^s/k).$ Note that $gv\in V_f^{c}(k^s)$ and $$(gg_\sigma c_\sigma^{-1}g^{-1})\cdot\gl{(gv)}{\sigma} = gv.$$ A direct computation then gives
$$(g_\sigma c_\sigma^{-1})\cdot\sigma(g_\tau c_\tau^{-1})\cdot(c_{\sigma\tau}g_{\sigma\tau}^{-1}) = g_\sigma\gl{g_\tau}{\sigma}g_{\sigma\tau}^{-1}.$$
The fact that $d_f$ is independent of the pure inner form suggests that $d_f=0$ might be sufficient for the existence of a rational orbit for {\it some} pure inner twist. Indeed, this is the case.

\begin{theorem}\label{thm:liftsuf}
Suppose that $f$ is a rational invariant, and that $G(k^s)$ acts transitively on $V_f(k^s)$ with abelian stabilizers.
Then $d_f=0$ in $H^2(k, G_f)$ if and only if there exists a pure inner form $G^{c}$ of $G$ such that $V^{c}_f(k)$ is nonempty. That is, the condition $d_f=0$ is necessary and sufficient for the existence of rational orbits for some pure inner twist of $G$.
In particular, when $H^1(k, G)= 1$, the condition $d_f=0$ in $H^2(k,G_f)$ is necessary and sufficient for the existence of rational orbits of $G(k)$ on $V_f(k)$.
\end{theorem}

\begin{proof}
Necessity has been shown in Proposition \ref{prop:liftnec} and the above computation. It remains to prove sufficiency.  Fix $v\in V_f(k^s)$ and $g_\sigma$ such that $g_\sigma \gl{v}{\sigma} = v$ for any $\sigma\in \text{Gal}(k^s/k).$ The idea of the proof is that if $d_f=0,$ then one can pick $g_\sigma$ so that $(\sigma\rightarrow g_\sigma)$ is a 1-cocycle and that rational orbits exist for the pure inner twist associated to this 1-cocycle.

Suppose $d_f = 0$ in $H^2(k, G_f)$. Then there exists a 1-cochain $(\sigma \rightarrow b_\sigma)$ with values in $G_f(k^s)$ such that
$$g_\sigma\gl{g_\tau}{\sigma} g_{\sigma\tau}^{-1} = \iota_v(b_\sigma\gl{b_\tau}{\sigma} b_{\sigma\tau}^{-1}) \quad\forall \sigma,\tau\in \mbox{Gal}(k^s/k).$$

\begin{lemma}\label{lem:prolemma}
There exists a $1$-cochain $e_\sigma$ with values in $G_v(k^s)$ such that $(\sigma \rightarrow e_\sigma g_\sigma )$ is a $1$-cocycle.
\end{lemma}

To see how Lemma \ref{lem:prolemma} implies Theorem \ref{thm:liftsuf}, we consider the twist of $G$ and $V$ using the 1-cocycle $c=(\sigma \rightarrow e_\sigma g_\sigma)\in H^1(k, G).$ Choose any $g\in \GL(V)(k^s)$ such that $g^{-1}\gl{g}{\sigma} = e_\sigma g_\sigma$ for any $\sigma\in\text{Gal}(k^s/k).$ Then $gv\in V^{c}_f(k)$. Indeed,
$$\gl{(gv)}{\sigma} = g e_\sigma g_\sigma \gl{v}{\sigma} = g e_\sigma v = gv \quad\forall\sigma\in\text{Gal}(k^s/k).$$

We now prove Lemma \ref{lem:prolemma}. Consider $e_\sigma = \iota_v(b_\sigma^{-1})$ for any $\sigma\in\text{Gal}(k^s/k).$ Since $g_\sigma \gl{v}{\sigma} = v,$ we have by \eqref{eq:canstab} and \eqref{eq:cangalstab} that
$$g_\sigma \gl{(\iota_v(b))}{\sigma} g_\sigma^{-1} = \iota_v(\gl{b}{\sigma})\quad\forall \sigma\in\text{Gal}(k^s/k),b\in G_f(k^s).$$
Hence for any $\sigma,\tau\in \text{Gal}(k^s/k),$ we have
\begin{eqnarray*}
(e_\sigma g_\sigma)\gl{(e_\tau g_\tau)}{\sigma}(e_{\sigma\tau} g_{\sigma\tau})^{-1}&=& \iota_v(b_\sigma^{-1}) g_\sigma\gl{(\iota_v(b_\tau^{-1}))}{\sigma} \gl{g_\tau}{\sigma} g_{\sigma\tau}^{-1}\iota_v(b_{\sigma\tau})\\
&=&\iota_v(b_\sigma^{-1})\iota_v(\gl{b_\tau^{-1}}{\sigma})g_\sigma\gl{g_\tau}{\sigma} g_{\sigma\tau}^{-1}\iota_v(b_{\sigma\tau})\\
&=&\iota_v(b_\sigma^{-1})\iota_v(\gl{b_\tau^{-1}}{\sigma})\iota_v(b_\sigma\gl{b_\tau}{\sigma} b_{\sigma\tau}^{-1})\iota_v(b_{\sigma\tau})\\ &=& 1
\end{eqnarray*}
where the last equality follows because $G_f(k^s)$ is abelian.
\end{proof}

\begin{corollary}\label{stcor}
Suppose that $f$ is a rational orbit and that $G(k^s)$ acts simply transitively on $V_f(k^s)$. Then there is a unique pure inner form $G^c$ of $G$
such that $V^{c}_f(k)$ is nonempty. Moreover, the group $G^c(k)$ acts simply transitively on $V_f^c(k)$.
\end{corollary}

\begin{proof}
Since $G_f = 1$, we have $H^2(k,G_f) = 0$ and so the cohomological obstruction $d_f$ vanishes. We conclude that rational orbits exist for some
pure inner twist $G^c$. Let $v_0\in V_f^c(k)$ denote any $k$-rational lift. Since $H^1(k,G_f) = 0$, the image of $\gamma: H^1(k,G^c_{v_0}) \rightarrow H^1(k,G^c)$ is a single point, and hence no other
pure inner twist has a rational orbit with invariant $f$. Since the kernel of $\gamma$ has cardinality $1$, there is a single orbit of $G^c(k)$ on $V^{c}_f(k)$.
\end{proof}

\section{Examples with trivial stabilizer}\label{trivstab}

In this section, we give several examples of representations $G \rightarrow \GL(V)$ over $k$ where there are stable orbits which
are determined by their invariants $f$ in $V/\!\!/G$ and which
have trivial stabilizer over $k^s$. Thus $G(k^s)$ acts simply transitively on the set $V_f(k^s)$. When $f$ is rational, Corollary~\ref{stcor} implies
that there is a unique pure inner form $G'$ of $G$ over $k$ for which $V'_f(k)$ is nonempty, and that $G'(k)$ acts simply transitively on $V'_f(k)$.

We will describe this pure inner form, using the following results on classical groups \cite{Inv}. Since $H^1(k,\GL(W))$ and $H^1(k,\SL(W))$ are both pointed sets with a single element, there are no nontrivial pure inner forms of $\GL(W)$ and $\SL(W)$. On the other hand, when the characteristic of $k$ is not equal to $2$ and $W$ is a nondegenerate quadratic space over $k$, the pointed set $H^1(k,\SO(W))$ classifies the quadratic spaces $W'$ with $\dim(W') = \dim(W)$ and $\disc(W') = \disc(W)$. The corresponding pure inner form is the group $G' = \SO(W')$. Similarly, if $W$ is a nondegenerate Hermitian space over the separable quadratic extension $E$ of $k$, then the pointed set $H^1(k,U(W))$ classifies Hermitian spaces $W'$ over $E$ with $\dim(W') = \dim(W)$, and the corresponding pure inner form of $G$ is the group $G' = U(W')$.

\subsection{$\SO(n+1)$ acting on $n$ copies of the standard representation}

In this subsection, we assume that $k$ is a field of characteristic not equal to $2$.

We first consider the action of the split group $G = \SO(W) = \SO(4)$ on three copies of the standard representation $V = W \oplus W \oplus W$. Let $q(w) = \langle w,w \rangle/2$ be the quadratic form on $W$ and let $v = (w_1,w_2,w_3)$ be a vector in $V$. The coefficients of the ternary quadratic form $f(x,y,z) = q(xw_1 + yw_2 + zw_3)$ give six invariant polynomials of degree $2$ on $V$, which freely generate the ring of polynomial invariants, and an orbit is stable if the discriminant $\Delta(f)$ of this quadratic form is nonzero in $k^s$. In this case, the group $G(k^s)$ acts simply transitively on $V_f(k^s)$. Indeed, the quadratic space $U_0$ of dimension $3$ with form $f$ embeds isometrically into $W$ over $k^s$, and the subgroup of $\SO(W)$ that fixes $U_0$ acts faithfully on its orthogonal complement, which has dimension $1$. The condition that the determinant of an
element in $\SO(W)$ is equal to $1$ forces it to act trivially on the orthogonal complement.

The set $V_f(k)$ is nonempty if and only if the quadratic form $f$ represents zero over $k$. Indeed, if $v = (w_1,w_2,w_3)$ is a vector in this orbit over $k$, then the vectors $w_1,w_2,w_3$ are linearly independent and span a $3$-dimensional subspace of $W$. This subspace must have a nontrivial intersection with
a maximal isotropic subspace of $W$, which has dimension $2$. Conversely, if the quadratic form $f$ represents zero, let $U_0$ be the $3$-dimensional quadratic space with this bilinear form, and $U$ the orthogonal direct sum of $U_0$ with a line spanned by a vector $u$ with $\langle u,u \rangle = \det(U_0)$. Then $U$ is a quadratic space of dimension $4$ and discriminant $1$  containing an isotropic line (from $U_0$). It is therefore split, and isomorphic over $k$ to the quadratic space $W$. Choosing an isometry $\theta: U \rightarrow W$, we obtain three vectors $(w_1,w_2,w_3)$ as the images of the basis elements of $U_0$, and this gives the desired element in $V_f(k)$. Note that $\theta$ is only well-defined up to composition by an automorphism of $W$, so we really obtain an orbit for the orthogonal group of $W$. Since the stabilizer of this orbit is a simple reflection, we obtain a single orbit for the subgroup $\SO(W)$.

If the form $f$ does not represent zero, let $W'$ be the quadratic space of dimension $4$ that is the orthogonal direct sum of the subspace $U_0$ of dimension $3$ with quadratic form $f$ and a nondegenerate space of dimension $1$, chosen so that the discriminant of $W'$ is equal to $1$. Then $G' = \SO(W')$ is the unique pure inner form of $G$ (guaranteed to exist by Corollary~\ref{stcor}) where $V'_f(k)$ is nonempty. The construction of an orbit for $G'$ is the same as above.

The same argument works for the action of the group $G = \SO(W) = \SO(n+1)$ on $n$ copies of the standard representation: $V = W \oplus W \oplus \cdots \oplus W$. The coefficients of the quadratic form $f(x_1,x_2, \ldots, x_n) = q(x_1w_1 + x_2 w_2 + \cdots + x_n w_n)$ give polynomial invariants of degree $2$, which freely generate the ring of invariants. The orbit of $v = (w_1,w_2, \ldots, w_n)$ is stable, with trivial stabilizer, if and only if the discriminant $\Delta(f)$ is nonzero in $k^s$. If $W'$ is the quadratic space of dimension $n + 1$ with $\disc(W') = \disc(W)$, that is the orthogonal direct sum of the space $U_0$ of dimension $n$ with quadratic form $f$ and a nondegenerate space of dimension $1$, then $G' = \SO(W')$ is the unique pure inner form with $V'_f(k)$ nonempty.

\subsection{$\SL(5)$ acting on $3$ copies of the representation $\wedge^2(5)$}

Let $k$ be a field of characteristic not equal to 2, \,$U$ a $k$-vector space of dimension 3, and $W$ a $k$-vector space of dimension 5.  In this subsection, we consider the action of $G=\SL(W)$ on $V=U\otimes \wedge^2W$.

Choosing bases for $U$ and $W$, we may identify $U(k)$ and $W(k)$ with $k^3$ and $k^5$, respectively, and thus $V(k)$ with $\wedge^2k^5\oplus\wedge^2k^5\oplus\wedge^2k^5$.
We may then represent elements of $V(k)$ as a triple $(A,B,C)$ of $5\times 5$ skew-symmetric matrices with entries in $k$.  For indeterminates $x$, $y$, and $z$, we see that the determinant of $Ax+By+Cz$ vanishes, being a skew-symmetric matrix of odd dimension.

To construct the $G$-invariants on $V$, we consider instead the $4\times 4$ principal sub-Pfaffians of $Ax+By+Cz$; this yields five ternary quadratic forms $Q_1,\ldots,Q_5$ in $x$, $y$, and $z$, which are generically linearly independent over $k$.  In basis-free terms, we obtain a $G$-equivariant map
\begin{equation}\label{firstmap}
U\otimes \wedge^2W \to \Sym^2 U\otimes W^\ast.
\end{equation}
Now an $\SL(W)$-orbit on $\Sym^2 U\otimes W^\ast$ may be viewed as a five-dimensional subspace of $\Sym^2U$; hence we obtain a natural $G$-equivariant map
\begin{equation}\label{secondmap}
 \Sym^2 U\otimes W^\ast \to \Sym^2U^\ast.
 \end{equation}
The composite map $\pi:U\otimes \wedge^2W \to \Sym^2U^\ast$ is thus also $G$-equivariant, but since $G$ acts trivially on the image of $\pi$, we see that the image of $\pi$ gives (a 6-dimensional space of) $G$-invariants, and indeed we may identify $V/\!\!/G$ with $\Sym^2U^\ast$.   A vector $v\in V$ is stable precisely when $\det(\pi(v))\neq 0$.

Now since $\SL(W)$ acts with trivial stabilizer on $W^\ast$, it follows that $\SL(W)$ acts with trivial stabilizer on $\Sym^2U\otimes W^\ast$ too.  Since the map (\ref{firstmap}) is $G$-equivariant, it follows that the generic stabilizer in $G(k)$ of an element in $V(k)$ is also trivial!

Since $\SL(W)$ has no other pure inner forms, by Corollary~\ref{stcor} we conclude that every $f\in \Sym^2U^\ast$ of nonzero determinant arises as the set of $G$-invariants for a unique $G(k)$-orbit on $V(k)$.

\subsection{$U(n-1)$ acting on the adjoint representation $ \frak u(n)$ of $U(n)$}

In this subsection, we assume that the field $k$ does not have characteristic $2$ and that $E$ is an \'etale $k$-algebra of rank $2$.
Hence $E$ is either a separable quadratic extension field, or the split algebra $k\times k$. Let $\tau$ be the nontrivial involution of $E$ that fixes $k$.

Let $Y$ be a free $E$-module of rank $n\geq 2$, and let
$$\langle ~,~ \rangle: Y \times Y \rightarrow E$$
be a nondegenerate Hermitian symmetric form on $Y$. In particular $\langle y,z \rangle = \gl{\langle z,y \rangle}{\tau}$.
Let $e$ be a vector in $Y$ with $\langle e,e \rangle \neq 0$, and let $W$ be the orthogonal complement of $e$ in $Y$.
Hence $Y = W \oplus Ee$. The unitary group $G = U(W) = U(n-1)$ embeds as the subgroup of $U(Y)$ that
fixes the vector $e$. In particular, it acts on the Lie algebra $\frak u(Y) = \frak u(n)$ via the restriction of the adjoint
representation.

Define the adjoint $T^*$ of an $E$-linear map $T: Y \rightarrow Y$ by the usual formula $\langle Ty, z \rangle = \langle y, T^*z \rangle$.
The elements of the group $U(Y)$ are the maps $g$ that satisfy $g^* = g^{-1}$. Differentiating this identity, we see that the elements
of the Lie algebra are those endomorphisms of $Y$ that satisfy $T + T^* = 0$. The group acts on the space of skew self-adjoint
operators by conjugation: $T \to gTg^{-1} = gTg^*$. If $T$ is skew self-adjoint and $\delta$ is an invertible element in $E$ satisfying
$\delta^{\tau} = -\delta$, then the scaled operator $\delta T$ is self-adjoint. Hence the adjoint representation of $U(Y)$ on its Lie algebra
is isomorphic to its action by conjugation on the vector space $V$, of dimension $n^2$ over $k$, consisting of the self-adjoint
endomorphisms $T: Y \rightarrow Y$. In this subsection, we consider the restriction of this representation to the subgroup $G = U(W)$.

The ring of polynomial invariants for $G = U(W)$ on $V$ is a polynomial ring, freely generated by the $n$ coefficients
$c_i(T)$ of the characteristic polynomial of $T$ (which are invariants for the larger group $U(Y)$) as well as the $n-1$ inner products $\langle e, T^j e \rangle$ for
$j = 1,2,\ldots, n-1$ (\cite[Lemma 3.1]{Z}). Note that all of these coefficients and inner products take values in $k$, as $T$ is self-adjoint. In particular, the space $V/\!\!/G$ is isomorphic to the affine space of dimension $2n-1$. Note that the inner products $\langle T^ie, T^je \rangle$ are all polynomial invariants for the action of $G$. Let $D$ be the invariant polynomial that is the determinant
of the $n \times n$ symmetric matrix with entries $\langle T^ie, T^je \rangle$ for $0 \leq i,j \leq n-1$. Clearly $D$ is nonzero if and only if the vectors $\{e,Te,T^2e,\ldots,T^{n-1}e\}$ form a basis for the space $Y$ over $E$. Rallis and Shiffman \cite[Theorem 6.1]{RS} show that the condition $D(f) \neq 0$ is equivalent to the condition that $G(k^s)$ acts simply transitively on the points of $V_f(k^s)$ . We can therefore conclude that when $D(f)$ is nonzero, there is a unique pure inner form $G'$ of $G = U(W)$ that acts simply transitively on the corresponding points in $V'_f(k)$, and that these spaces are empty for all other pure inner forms. To determine the pure inner form $G'=U(W')$ for which $V'_f(k)$ is nonempty, it suffices to determine the Hermitian space $W'$ over $E$ of rank $n-1$. The rational invariant $f$
determines the inner products $\langle T^ie, T^je \rangle$, and hence a Hermitian structure on $Y' = Ee + E(Te) + \cdots + E(T^{n-1}e)$. Since the nonzero value $\langle e,e \rangle$ is fixed, this gives the Hermitian structure on its orthogonal complement $W'$ in $Y'$, and hence the pure inner form $G'$ such that $V'_f(k)$ is nonempty.

When the algebra $E$ is split, the Hermitian space $Y = X + X^{\vee}$ decomposes as the direct sum of an $n$-dimensional vector space
$X$ over $k$ and its dual. The group $U(Y)$ is isomorphic to $\GL(X) = \GL(n)$. The vector $e$ gives a nontrivial
vector $x$ in $X$ as well as a nontrivial functional $f$ in $X^{\vee}$ with $f(x) \neq 0$. Let $X_0$ be the kernel of $f$, so $X = X_0 + kx$.
The subgroup $U(W)$ is isomorphic to $\GL(X_0) = \GL(n-1)$. In this case, the representation of $U(W)$ on the space of self-adjoint
endomorphisms of $Y$ is isomorphic to the representation of $G = \GL(n-1)$ by conjugation on the space $V = \End(X)$ of all
$k$-linear endomorphisms of $X$. Since $\GL(n-1)$ has no pure inner forms, Corollary~\ref{stcor} implies that $\GL(n-1)$ acts simply transitively
on the points of $V_f(k)$ whenever $D(f) \neq 0$.

Once we have chosen an invertible element $\delta$ in $E$ of trace zero, the rational invariants for the action of $U(W) = U(n-1)$ on the Lie algebra of $U(n)$ match the rational invariants for the action of $\GL(X) = \GL(n-1)$ on the Lie algebra of $\GL(n)$. Since the stable orbits for the pure inner forms $U(W')$ and $\GL(X)$ are determined by these rational invariants, we obtain a matching of orbits. 
This gives a natural explanation for the matching of    
orbits that plays an important role in the work of 
Jacquet and Rallis~\cite{JR} on the relative trace formula, where they establish a comparison of the corresponding orbital integrals,
and in the more recent work of  Wei Zhang \cite{Z} on the global conjecture of Gan, Gross, and Prasad~\cite{GGP}.

\section{Examples with nontrivial stabilizer and nontrivial obstruction}\label{nontrivstab}

In this section, we will provide some examples of representations with a nontrivial abelian stabilizer~$G_f$, and calculate the obstruction class $d_f$ in $H^2(k,G_f)$. The first example is a simple modification of a case we have already considered, namely, the non-faithful representation $V$ of $\Spin(W) = \Spin(n+1)$ on $n$ copies of the standard representation $W$ of the special orthogonal group $\SO(W)$. In this case, the stabilizer $G_f$ of the stable orbits is the center $\mu_2$. We will also describe the stable orbits for the groups $G=\SL(W)$ and $H=\SL(W)/\mu_2$ acting on the representation $V = \Sym_2W^* \oplus \Sym_2W^*$. (The group $H$ exists and acts when the dimension of $W$ is even.) In these cases, the stabilizer $G_f$ is a finite elementary abelian $2$-group, related to the 2-torsion in the Jacobian of a hyperelliptic curve.

\subsection{$\Spin(n+1)$ acting on $n$ copies of the standard representation of $\SO(n+1)$}\label{spinrep}

In this subsection, we reconsider the representation $V = W^n$ of $\SO(W)$ studied in $\S 3.1$. There we saw that the orbits of vectors $v = (w_1,w_2,\ldots, w_n)$, where the quadratic form $f = q(x_1w_1 + x_2w_2 + \cdots+ x_nw_n)$ has nonzero discriminant, have trivial stabilizer. If we consider $V$ as a representation of the two-fold covering group $G = \Spin(W)$, then these orbits have stabilizer $G_f = \mu_2$.

In the former case, we found that the unique pure inner form $\SO(W')$ for which $V_f'(k)$ is nonempty corresponded to the quadratic space $W'$ of dimension $n+1$ and $\disc(W') = \disc(W)$ that is the orthogonal direct sum of the subspace $U_0$ with quadratic form $f$ and a nondegenerate space of dimension $1$.
The group $\Spin(W')$ will have orbits with invariant $f$, but this group may {\it not} be a pure inner form of the group $G = \Spin(W)$. If it is not a pure inner form, the invariant $d_f$ must be non-trivial in $H^2(k,G_f)$.

Assume, for example, that the orthogonal space $W$ is split and has odd dimension $2m+1$, so that the spin representation $U$ of $G =\Spin(W)$ of dimension $2^m$ is defined over $k$. Then a necessary and sufficient condition for the group $G' = \Spin(W')$ to be a pure inner form of $G$ is that the even Clifford algebra $C^+(W')$ of
$W'$ is a matrix algebra over $k$. In this case, the spin representation $U'$ of $G'$ can also be defined over $k$. Hence the obstruction $d_f$ is given by the Brauer class of the even Clifford algebra of the space $W'$ determined by $f$.  Note that the even Clifford algebra $C^+(W')$ has an anti-involution, so its Brauer class has order 2 and lies in the group $H^2(k,G_f) = H^2(k,\mu_2)$. 

\subsection{$\SL_n$ acting on $\Sym_2(n) \oplus \Sym_2(n)$}

Let $k$ be a field of characteristic not equal to $2$ and let $W$ be a vector space of dimension $n$ over $k$. Let $e$
be a basis vector of the one-dimensional vector space $\wedge^n W$. The group $G = \SL_n$ acts linearly on $W$ and trivially on $\wedge^nW$.

The action of $G$ on the space $\Sym_2W^*$ of symmetric bilinear forms $\langle v,w\rangle$ on $W$ is given by
the formula
$$ g \cdot \langle v,v' \rangle = \langle gv, gv' \rangle$$
This action preserves
the discriminant of the bilinear form $A = \langle\:\,,\: \rangle$, which is defined by the formula:
$$\disc(A) = (-1)^{n(n-1)/2} \langle e,e \rangle_n.$$
Here $\langle\:\,,\: \rangle_n$ is the induced symmetric bilinear form on $\wedge^n (W)$. If $\{w_1,w_2, \ldots, w_n\}$ is
any basis of $W$ with $w_1 \wedge w_2 \wedge \ldots \wedge w_n = e$, then $\langle e, e \rangle_n = \det (\langle w_i,w_j\rangle)$.
The discriminant is a polynomial of degree $n = \dim(W)$ on $\Sym_2W^*$  which
freely generates the ring of $G$-invariant polynomials.

Now consider the action of $G$ on the representation $V = \Sym_2W^* \oplus \Sym_2W^*$. If
$A = \langle\:\,,\: \rangle_A$ and $B = \langle\:\,,\: \rangle_B$ are two symmetric bilinear forms on $W$, we define the binary form of degree $n$
over $k$ by the formula
$$f(x,y) = \disc(xA - yB) = f_0x^n + f_1x^{n-1}y + \cdots + f_ny^n.$$
The coefficients of this form are each polynomial invariants of degree $n$ on $V$, and the $n+1$ coefficients $f_j$ freely generate the ring
of polynomial invariants for $G$ on $V$. (This will follow from our determination of the orbits of $G$ over $k^s$ in Theorem \ref{stable orbits}.)
We call $f(x,y)$ the {\em invariant binary form} associated to (the orbit of) the vector $v = (A,B)$.

The discriminant $\Delta(f)$ of the binary form $f$ is defined by writing $f(x,y) = \prod (\alpha_ix - \beta_iy)$ over the
algebraic closure of $k$ and setting
$$\Delta(f) = \prod_{i < j}(\alpha_i\beta_j - \alpha_j\beta_i)^2.$$
Then $\Delta(f)$ is a homogeneous polynomial of degree $2n - 2$ in the coefficients $f_j$, so is a polynomial invariant of degree $2n(n-1)$ on $V$.
For example, the binary quadratic form $ax^2 + bxy + cy^2$ has discriminant $\Delta = b^2 - 4ac$ and the binary cubic form $ax^3 + bx^2y + cxy^2 + dy^3$ has discriminant
$\Delta = b^2c^2 + 18abcd - 4ac^3 - 4b^3d - 27a^2d^2$.

The first result shows how the invariant form and its discriminant determine the stable orbits for $G$ on $V$ over $k^s$.

\begin{theorem}
\label{stable orbits}
Let $k^s$ be a separable closure of $k$, and let $f(x,y)$ be a binary form of degree $n$ over $k^s$ with
$f_0 \neq 0$ and $\Delta(f) \neq 0$. Then there are vectors $(A,B)$ in $V(k^s)$ with invariant form $f(x,y)$,
and these vectors all
lie in a single orbit for $G(k^s)$. This orbit is closed, and the stabilizer of any vector in the orbit is
an elementary abelian $2$-group of order $2^{n-1}$.
\end{theorem}

To begin the proof, we make a simple observation. Let $A$ and $B$ denote two symmetric bilinear forms
on $W$ over $k^s$ with $\disc(xA - yB) = f(x,y)$. Then both $A$ and $B$ give $k^s$-linear maps $W \rightarrow W^*$. Our
assumption that $f_0$ is nonzero implies that the linear map $A: W \rightarrow W^*$ is an isomorphism, so we
obtain an endomorphism $T = A^{-1}B : W \rightarrow W$. The fact that both $A$ and $B$ are symmetric with respect to transpose implies
that $T$ is self-adjoint with respect to the bilinear form $\langle\:\,,\: \rangle_A$ on $W$.

Write $f(x,1) = f_0g(x)$ with $g(x)$ monic of degree $n$. The characteristic polynomial $\det(xI - T)$
is equal to the monic polynomial $g(x)$, and our assumption that the discriminant of $f(x,y)$ is nonzero in $k$
implies that the polynomial $g(x)$ is separable. Hence the endomorphism $T$ of $V$ is regular and semisimple. The group
$G(k^s)$ acts transitively on the bilinear forms with discriminant $f_0$, and the stabilizer of $A$ is the orthogonal
group $\SO(W,A)$. Since the group $\SO(W,A)(k^s)$ acts transitively on the self-adjoint operators $T$ with a fixed separable
characteristic polynomial $g(x)$, there is a single $G(k^s)$-orbit on the vectors $(A,B)$ with invariant form $f(x,y)$.
The stabilizer is the centralizer of $T$ in $\SO(W,A)$, which is an elementary abelian $2$-group of order $2^{n-1}$. For
proofs of these assertions, see \cite[Prop.~4]{BG}.

\bigskip

Having classified the stable orbits of $G$ on $V$ over the separable closure, we now turn to the problem of
classifying the orbits with a fixed invariant polynomial $f(x,y)$ over $k$.

\begin{theorem}
\label{rational orbits}
Let $f(x,y) = f_0x^n + f_1 x^{n-1}y+\cdots + f_ny^n$ be a binary form of degree $n$ over $k$ whose discriminant $\Delta$
and leading coefficient $f_0$ are both nonzero in $k$. Write $f(x,1) = f_0g(x)$ and let $L$ be the \'etale algebra
$k[x]/(g)$ of degree $n$ over $k$. Then there is
a canonical bijection {\em (}constructed below{\em )} between the set of orbits $(A,B)$ of $G(k)$
on $V(k)$ having invariant binary form $f(x,y)$ and the equivalence classes of pairs $(\alpha, t)$ with $\alpha \in L^\times$ and
$t \in k^\times$, satisfying $f_0N(\alpha) = t^2$. The pair $(\alpha, t)$ is equivalent to the pair $(\alpha^*, t^*)$ if there is an element $c \in L^\times$ with $c^2 \alpha^* = \alpha$ and $N(c)t^* = t$.

The group scheme $G_f$ obtained by descending the stabilizers $G_{A,B}$ for $(A,B)\in V_f(k^s)$ to $k$ is the finite abelian group scheme $(\Res_{L/k} \mu_2)_{N = 1}$ of order $2^{n-1}$ over $k$.
\end{theorem}

As a corollary, we see that the set of orbits with invariant form $f(x,y)$ is nonempty if and only if the element $f_0 \in k^\times$ lies in the subgroup $N(L^\times)k^{\times 2}$. In this case, we obtain a surjective map (by forgetting $t$) from the set of orbits to the set $(L^\times/L^{\times2})_{N \equiv f_0}$, where the subscript indicates that the norm is congruent to $f_0$ in the group $k^\times/k^{\times 2}$. This map is a bijection when there is an element $c \in L^\times$ that satisfies $c^2 = 1$ and $N(c) = -1$. Such an element $c$ will exist if and only if the polynomial $g(x)$ has a monic factor of odd degree over $k$. If no such element $c$ exists, then the two orbits $(\alpha,t)$ and $(\alpha, -t)$ are distinct and map to the same class $\alpha$ in $(L^\times/L^{\times2})_{N \equiv f_0}$. In that case, the map is two-to-one.

When $n$ is odd, the set of orbits is always nonempty and has a natural base point $(\alpha, t) = (f_0, f_0^{(n+1)/2})$. Using this base point, and the existence of an element $c$ with $c^2 = 1$ and $N(c) = -1$, we can identify the set of orbits with the group $(L^\times/L^{\times2})_{N \equiv 1}$.  When $n$ is even, $f_0$ may not lie in the subgroup $N(L^\times)k^{\times2}$ of $k^\times$. In this case,
there may be no orbits over $k$ with invariant polynomial $f(x,y)$. For example, when $n = 2$ there are no orbits over $\mathbb R$ with invariant
form $f(x,y) = -x^2 - y^2$. Even when orbits exist, there is no natural base point and we can not identify the orbits with the elements of a group.

There is a close relation between the existence of an orbit with invariant $f(x,y)$ in the even case $n = 2g+2$ and the arithmetic of the smooth hyperelliptic curve $C$ of genus $g$ over $k$ with equation $z^2 = f(x,y)$. For example, every $k$-rational point $P = (u,1,v)$ on $C$ with $v \neq 0$ (so $P$ is not a Weierstrass point) gives rise to an orbit~\cite[\S2]{genhyper}. Indeed, write $f(x,1) = f_0 \cdot g(x)$ and let $\theta$ be the image of $x$ in the algebra $L = k[x]/(g(x))$. The orbit associated to $P$ has
$\alpha = u - \theta \in L^\times$ and $t = v \in k^\times$. Then $N(\alpha) = g(u)$, so $t^2 = f_0\cdot N(\alpha)$.  This is the association used in \cite{genhyper} to show that most hyperelliptic curves
over $\Q$ have no rational~points.

\vspace{.17in}

\noindent \textbf{Proof of Theorem \ref{rational orbits}:}
Assume that we have a vector $(A,B)$ in $V(k)$ with $\disc(xA - yB) = f(x,y)$.
Using the $k$-linear maps $W \rightarrow W^*$ given by the bilinear forms $A$ and $B$ and the assumption that
$f_0$ is nonzero, we
obtain an endomorphism $T = A^{-1}B : W \rightarrow W$ which is self-adjoint for the pairing $\langle,\rangle_A$ and
has characteristic polynomial $g(x)$. Since $\Delta(f)$ is nonzero, the polynomial $g(x)$ is separable and $W$
has the structure of a free $L = k[T] = k[x]/(g)$ module of rank one. Let $\beta$ denote the image
of $x$ in $L$, and let $\{1,\beta, \beta^2, \cdots, \beta^{n-1}\}$ be the corresponding power basis of $L$ over $k$.

The $k$-bilinear forms $A$ and $B$ both arise as the traces of $L$-bilinear forms on the rank one $L$ module $W$. Choose
a basis vector $m$ of $W$ over $L$ and consider the $k$-linear map $L \rightarrow k$ defined by
$\lambda \to  \langle m, \lambda m \rangle_A$. Since $g(x)$ is separable, the element $g'(\beta)$ is a unit
in $L$ and the trace map from $L$ to $k$ is nonzero. Hence there is a unique element $\kappa$ in $L^\times$ such that
$$\langle m,\lambda m \rangle_A = \Trace(\kappa \lambda/g'(\beta))$$
for all $\lambda$ in $L$. Since all elements of $L$ are self-adjoint with respect to the form $\langle , \rangle_A$, we find that the formula
$$\langle \mu m,\lambda m\rangle_A = \Trace(\kappa\mu\lambda/g'(\beta))$$
holds for all $\mu$ and $\lambda$ in $L$. Since the discriminant $f_0$ of the bilinear form $\langle, \rangle_A$ is nonzero in $k$,
we conclude that $\kappa$ is a unit in the algebra $L$, so is an element of the group $L^\times$. We define $\alpha = \kappa^{-1} \in  L^\times,$
so that
$$\langle \mu m,\lambda m \rangle_A = \Trace(\mu\lambda/\alpha  g'(\beta)).$$
A famous formula due to Euler \cite[Ch III, \S 6]{S2} then shows that for all $\mu$ and $\lambda$ in $L$,
the value $\langle \mu m,\lambda m \rangle_A$ is the coefficient of $\beta^{n-1}$ in the basis expansion
of the product $\mu \lambda/\alpha$. It follows that the value $\langle \mu m,\lambda m \rangle_B$ is the
coefficient of $\beta^{n-1}$ in the basis expansion of the product $\beta \mu \lambda/\alpha$.

We define the element $t \in k^{\times}$ by the formula
$$t(m\wedge \beta m \wedge \beta^2 m \wedge \ldots \wedge\beta^{n-1}m) = e$$
in the one-dimensional vector space $\wedge^n(W)$. Then $\langle e,e \rangle_n = t^2
\det (\langle \beta^i m, \beta^j m \rangle_A)$. Since $\langle e,e \rangle_n = (-1)^{n(n-1)/2}f_0$ and $\det (\langle \beta^i m, \beta^j m \rangle_A)=(-1)^{n(n-1)/2} N(\alpha)^{-1}$, we have that $t^2 = f_0N(\alpha)$.

We have therefore associated to the binary form $f(x,y)$ an \'etale algebra $L$, and to the
vector $(A,B)$ with discriminant $f(x,y)$ an element $\alpha \in L^\times$ and an element $t \in k^\times$ satisfying $t^2 = f_0N(\alpha)$. The definition of $\alpha$ and $t$
required the choice of a basis vector $m$ for $W$ over $L$.
If we choose instead $m^* = cm$ with $c$ in $L^\times$, then $\alpha = c^2\alpha^*$ and $t = N(c)t^*$. Hence
the vector $(A,B)$ only determines the equivalence class of the pair $(\alpha, t)$ as defined above.

It is easy to see that every equivalence class $(\alpha, t)$ determines an orbit.
Since the dimension $n$
of $L$ over $k$ is equal to the dimension $n$ of $W$, we can choose a linear isomorphism
$\theta: L \rightarrow W$ that maps the element $1 \wedge \beta \wedge \beta^2 \ldots\wedge\beta^{n-1}$ in $\wedge^n(L)$
to the element $t^{-1}e$ in $\wedge^n(V)$. Every other isomorphism
with this property has the form $h\theta$, where $h$ is an element in the subgroup $G = \SL(W)$.
Using $\theta$ we define two bilinear forms on $W$:
$$\langle \theta(\mu),\theta(\lambda) \rangle_A = \Trace (\mu\lambda/(\alpha g'(\beta)))$$
$$\langle \theta(\mu),\theta(\lambda) \rangle_B = \Trace (\beta \mu\lambda/(\alpha g'(\beta))).$$
The $G(k)$-orbit of the vector $(A,B)$ in $V(k)$ is well-defined and has invariant polynomial $f(x,y)$.

To complete the proof, we need to determine the stabilizer of a point $(A,B)\in V(k^s)$ in an orbit
with binary form $f(x,y)$. Let $L^s=k^s[x]/(g(x))$ denote the $k^s$-algebra of degree $n$.
Since the bilinear form $\langle\:\,,\: \rangle_A$ is nondegenerate,
the stabilizer of $A$ in $G$ is the special orthogonal group $\SO(W,A)$ of this form. The stabilizer of $B$
in the special orthogonal group $\SO(W,A)$ is the subgroup of those $g$ that commute with the self-adjoint
transformation $T$. Since $T$ is regular and semisimple, the centralizer of $T$ in $GL(W)$ is
the subgroup $k^s[T]^\times = L^{s\times}$, and the operators in $L^{s\times}$ are all self-adjoint. Hence the intersection
of $L^{s\times}$ with the special orthogonal group $\SO(W,A)(k^s)$ consists of those elements $g$ that are simultaneously
self-adjoint and orthogonal, so consists of those elements $g$ in $L^{s\times}$ with $g^2 = 1$ and $N(g) = 1$. The same
argument works over any $k^s$-algebra $E$. The elements in $G(E)$ stabilizing $(A,B)$
are the elements $h$ in $(E \otimes L^s)^\times$ with $h^2 = 1$ and $N(h) = 1$. Hence the stabilizer $G_{A,B}$ is isomorphic to
to the finite \'etale group scheme $(\Res_{L^s/k^s} \mu_2)_{N = 1}$ over $k^s$.

To show that these group schemes descend to $(\Res_{L/k}\mu_2)_{N=1}$, it remains to construct isomorphisms $\iota_v:(\Res_{L/k}\mu_2)_{N=1}(k^s)\rightarrow G_v$ compatible with the descent data for every $v\in V_f(k^s)$, i.e., satisfying \eqref{eq:canstab} and \eqref{eq:cangalstab}. Let $\alpha_1,\ldots,\alpha_n\in k^s$ denote the roots of $g(x)$. For any $i=1,\ldots,n$, define
$$h_i(x) = \frac{g(x)}{x-\alpha_i},\qquad g_i(x) = 1 - 2\frac{h_i(x)}{h_i(\alpha_i)}.$$
For any linear operator $T$ on $W$ with characteristic polynomial $g(x)$, the operator $g_i(T)$ acts as $-1$ on the $\alpha_i$-eigenspace of $T$ and acts trivially on all the other eigenspaces. Then for any $v=(A,B)\in V_f(k^s)$, the map $\iota_v$ sends an $n$-tuple $(m_1,\ldots,m_n)$ of 0's or 1's, such that $\sum m_i$ is even, to $$\iota_v(m_1,\ldots,m_n)=\prod_{i=1}^n g_i(T)^{m_i},$$ where $T=A^{-1}B$ as before.
{\hfill$\Box$ \vspace{2 ex}}

In \cite{Wood1}, Wood has classified the elements of the representation  $\Sym_2 R^n\oplus \Sym_2 R^n$, for any base ring (or even any base scheme) $R$, in terms of suitable algebraic data involving ideals classes of ``rings of rank $n$'' over $R$; see \S\ref{intorbits} for more details on the case $R=\Z$.  The special case where $R$ is a field, and a description of the resulting orbits under the action of $\SL_n(R)$, is given by Theorem~\ref{rational orbits}.

\subsection{Some finite group schemes and their cohomology}

To give a cohomological interpretation of Theorem \ref{rational orbits} and to make preparations for the study of the orbits of the action of $\SL_n/\mu_2$ on $\Sym^2(n)\oplus\Sym^2(n)$ in the next two subsections, we collect some results on the cohomology of $\Rlk\mu_2$ and other closely related finite group schemes.

Fix an integer $n \geq 1$, and consider the action of the symmetric group $S_n$ on the vector space
$N = (\mathbb Z/ 2\mathbb Z)^n$ by permutation of the natural basis elements $e_i$. The nondegenerate symmetric
bilinear form
$$\langle n, m \rangle = \sum n_im_i$$
is $S_n$-invariant. We have the stable subspace $N_0$ of elements with $\sum n_i = 0$, and on this
subspace the bilinear form is alternating. It is also nondegenerate when $n$ is odd.

When $n$ is even the
radical of the form on $N_0$ is the one-dimensional subspace $M$ spanned by the vector
$n = (1,1,\ldots,1)$, and we obtain a nondegenerate alternating pairing
$$N_0 \times N/M \rightarrow \mathbb Z/2\mathbb Z.$$
This induces an alternating duality which is $S_n$-invariant on the subquotient $N_0/M$.

We want to translate these results on finite elementary abelian $2$-groups with an action of $S_n$ to finite \'etale group schemes over
a field $k$ whose characteristic is not equal to $2$. Let $L$ be an \'etale $k$-algebra of rank $n$,
and let $R$ be the finite group scheme $\Res_{L/k}\mu_2$. Let $k^s$ be a fixed separable closure of $k$. The Galois group
of $k^s$ over $k$ permutes the $n$ distinct homomorphisms $L \rightarrow k^s$, and this
determines a homomorphism $\Gal(k^s/k) \rightarrow S_n$ up to conjugacy. We have an
isomorphism $R(k^s) \cong N$ of $\Gal(k^s/k)$ modules. If $L = k[x]/g(x) = k[\beta]$ with
$g(x)$ monic and separable of degree $n$, then the distinct homomorphisms $L \rightarrow k^s$
are obtained by mapping $\beta$ to the distinct roots $\beta_i$ of $g(x)$ in $k^s$. Hence the points of $R$ over
an extension $K$ of $k$ correspond bijectively to the monic
factors $h(x)$ of $g(x)$ over $K$.

Let $R_0 = (\Res_{L/k}\mu_2)_{N=1}$ be the subgroup scheme
of elements of norm $1$ to $\mu_2$. The above isomorphism maps $R_0(k^s)$ to the Galois module $N_0$, and the
points of $R_0$ over an extension $K$ correspond to the monic factors $h(x)$ of $g(x)$ of even degree
over $K$.

The diagonally embedded $\mu_2 \rightarrow R$ corresponds to the trivial Galois submodule
$M$ of $N$, and the points of $R/\mu_2$ over $K$ correspond to the monic factorizations $g(x) = h(x)j(x)$ that are rational over $K$. This means that either $h(x)$ and $j(x)$ have coefficients in $K$, or that they have conjugate coefficients in some quadratic extension of $K$.

When $n$ is even, the subgroup $\mu_2$ of $R$ is actually a subgroup of $R_0$. The points of $R_0/\mu_2$ over $K$ correspond to the monic factorizations $g(x) = h(x)j(x)$ of even degree
that are rational over $K$.

Since the pairings defined above are all $S_n$-invariant, we obtain Cartier dualities
$$R \times R \rightarrow \mu_2,$$
$$R_0 \times R/\mu_2 \rightarrow \mu_2.$$
Since the Cartier dual of $R_0$ is the finite group scheme
$R/\mu_2$, we obtain a cup product pairing
$$H^2(k, R_0) \times H^0(k, R/\mu_2) \longrightarrow H^2(k,\mathbb G_m)[2] = H^2(k, \mu_2).$$

When $n$ is odd, we obtain an alternating duality on $R_0 \cong R/\mu_2$. When $n$
is even, we obtain an alternating duality
$$R_0/\mu_2 \times R_0/\mu_2 \rightarrow \mu_2.$$

We now consider the Galois cohomology of these \'etale group schemes. For $R = \Res_{L/k}\mu_2$ we
have
$$H^0(k,R) = L^\times[2],\quad H^1(k,R) = L^\times/L^{\times 2},\quad H^2(k,R) = \mathrm{Br}(L)[2].$$
For $R_0 = (\Res_{L/k}\mu_2)_{N = 1}$, we have $H^0(k,R_0) = L^\times[2]_{N = 1}$ and the long exact sequence in cohomology gives an exact sequence
\begin{equation}\label{eq:dfone}
1 \to \langle\pm1\rangle /N(L^\times[2]) \to H^1(k,R_0) \to L^\times/L^{\times 2} \to k^\times/k^{\times 2} \to H^2(k,R_0) \to \mathrm{Br}(L)[2].
\end{equation}
The group $H^1(k,R_0)$ maps surjectively to the subgroup $(L^\times/L^{\times 2})_{N \equiv 1}$ of elements in $L^\times/L^{\times 2}$ whose norm to $k^\times/k^{\times 2}$ is a square. The kernel of this map has order one if $-1$ is the norm of an element of $L^\times [2]$, or equivalently if $g(x)$ has a factor of odd degree. If $g(x)$ has no factor of odd degree, then the kernel has order two.

This computation allows us to give a cohomological interpretation to Theorem \ref{rational orbits}. For each rational invariant $f(x,1) = f_0g(x)$ with nonzero $\Delta(f)$ and $f_0$, the stabilizer $G_f$ is isomorphic to the finite group scheme $(\Res_{L/k} \mu_2)_{N = 1} =~R_0$. The quotient group $k^\times/k^{\times2}(NL^\times)$ is the kernel of the map from  $H^2(k, R_0)$ to $H^2(k,R)$.
In Theorem \ref{thm:dfG}, we will show that the class of $f_0 \in k^\times/k^{\times 2}N(L^\times)$ maps to the class $d_f \in H^2(k,R_0)$ defined in \S\ref{sec:cohomobst}. Since $H^1(k,\SL_n)=0,$ by Theorem \ref{thm:liftsuf} the nontriviality of $d_f$ in $H^2(k,R_0)$ is the only
obstruction to the existence of an $\SL_n(k)$-orbit with invariant form $f(x,y)$. This gives another proof that rational orbits with invariant $f(x,y)$ exist if and only if $f_0\in N(L^\times)k^{\times2}.$ When the class $d_f$ vanishes, the orbits of $\SL_n(k)$ with this rational invariant $f$ form a principal
homogeneous space for the group $H^1(k,R_0)$.

\subsection{$\SL_n/\mu_2$ acting on $\Sym_2(n) \oplus \Sym_2(n)$}

When the dimension $n$ of $W$ is odd, we have obtained a bijection from the set of orbits for $\SL(W)$ with invariant $f$ to the elements of the group $(L^\times/L^{\times2})_{N \equiv 1}$

In this section, we consider the more interesting situation when the dimension $n$ of $W$ is even. In this case, the central subgroup $\mu_2$ in $\SL(W)$ acts
trivially on $V = \Sym_2W^* + \Sym_2W^*$, and we can consider the orbits of the group $H = \SL(W)/\mu_2$ on $V$ over $k$. Since $H^1(k,\SL(W)) = 1$, the group of $k$-rational points of $H$ lies in the exact sequence
$$ 1 \to \SL(W)(k)/\langle\pm1\rangle \to H(k) \to k^\times/k^{\times 2} \to 1.$$
 A representative in $H(k)$ of the coset of $d$ in $k^\times/k^{\times2}$ can be obtained as follows. Lift $d\in k^\times/k^{\times2}$ to an element $d \in k^\times$ and let $K = k(\sqrt d)$ be the corresponding quadratic extension. Let $\tau$ be the nontrivial involution of $K$ over $k$ and let $g(d)$ be any element of $\SL(W)(K)$ whose conjugate $\gl{(g(d))}{\tau}$ is equal to $-g(d)$. For example, one can take a diagonal matrix with $n/2$ entries equal to $\sqrt d$ and $n/2$ entries equal to $1/\sqrt d$. Then the image of $g(d)$ in the quotient group $H(K)$ gives a rational element in $H(k)$. The elements $g(d)$ for $d$ in $k^\times/k^{\times2}$ give coset representatives for the subgroup $\SL(W)(k)/ \langle \pm 1 \rangle$ of $H(k)$.

If $v$ is any vector in $V_f(k)$ and $d$ represents a coset of $k^\times/k^{\times2}$, then
$$\gl{(g(d)(v))}{\tau} = \gl{(g(d))}{\tau} (v) = -g(d) (v) = g(d)(v)$$
so the vector $g(d)(v)$ is also an element of $V_f(k)$. Since the coset of $g(d)$ is well-defined, and $g(d)^2$ is an element of $\SL(W)(k)$, we see that the action of $g(d)$ gives an involution (possibly trivial) on the orbits of $\SL(W)(k)= G(k)$ on $V_f(k)$.

We have seen that the orbits of $G(k)$ with invariants $f(x,y)$ are determined by two invariants: $\alpha \in L^\times$ and $t \in k^\times$ that satisfy $f_0N(\alpha) = t^2$. The pair $(\alpha, t)$ is equivalent to the pair $(c^2\alpha, N(c) t)$. Under this bijection, the element represented by $g(d)$ in $H(k)$ maps the equivalence class of $(\alpha, t)$ to the equivalence class $(d \alpha, d^{n/2} t)$. This gives the following result.

\begin{theorem}\label{thm:orbitH}
Assume that $n$ is even and let $f(x,y) = f_0x^n + f_1x^{n-1}y+\cdots + f_ny^n$ be a binary form of degree $n$ over $k$ whose discriminant $\Delta$
and leading coefficient $f_0$ are both nonzero in $k$. Write $f(x,1) = f_0g(x)$ and let $L$ be the \'etale algebra
$k[x]/(g)$ of degree $n$ over $k$.
Then there is
a bijection between the set of orbits $(A,B)$ of $H(k)$
on $V(k)$ having invariant binary form $f(x,y)$ and the set of equivalence classes of pairs $(\alpha, t)$ with $\alpha \in L^\times$ and
$t \in k^\times$ satisfying $f_0N(\alpha) = t^2$. The pair $(\alpha, t)$ is equivalent to the pair $(\alpha^*, t^*)$ if there is an element $c \in L^\times$ and an element $d \in k^\times$ with $c^2 d \alpha^* = \alpha$ and $N(c) d^{n/2} t^* = t$.

The group scheme $H_f$ obtained by descending the stabilizers $H_{A,B}$ for $(A,B)\in V_f(k^s)$ to $k$ is finite abelian group scheme $(\Res_{L/k} \mu_2)_{N = 1}=R_0/\mu_2$ of order $2^{n-2}$ over $k$.

\end{theorem}

\noindent Theorem \ref{thm:orbitH} implies that orbits for $H(k)$ exist with invariant binary form $f(x,y)$ if and only if the leading coefficient $f_0$ lies in the subgroup $k^{\times 2}N(L^\times)$ of $k^\times$. When orbits do exist, we can associate to each $H(k)$-orbit the class of $\alpha$ in the set
$$(L^\times/ L^{\times 2}k^\times)_{N \equiv f_0}.$$
This is a surjective map, which is a bijection when there are elements $c \in L^\times$ and $d \in k^\times$ satisfying $c^2 d = 1$ and $N(c) d^{n/2} = -1$. Such a pair $(c,d)$ exists if and only if the monic polynomial $g(x)$ has an odd factorization over $k$. If $g(x)$ has a rational factor of odd degree, then there is a pair with $c^2 = 1$ and $d = 1$. On the other hand, if $g(x)$ has no rational factor of odd degree, but has a rational factorization, then $n/2$ is odd and the factorization occurs over the unique quadratic extension $K = k(\sqrt d)$ which is a sub-algebra of $L$. If $g(x)$ has no odd factorization, the two orbits
$(\alpha, t)$ and $(\alpha, -t)$ are distinct and the surjective map from the set of $H(k)$-orbits to the set $(L^\times/ L^{\times 2}k^\times)_{N \equiv f_0}$  is two-to-one.

We can also reinterpret this result in terms of the Galois cohomology of the stabilizer $H_f~=~R_0/\mu_2$. We assume that there exists rational $(A,B)\in V(k)$ with invariant binary form $f(x,y)$. In the next subsection, we study the obstruction to this existence. The set of rational orbits with invariant $f$ is in bijection with the kernel of the composite map $\gamma: H^1(k,H_{A,B}) \rightarrow H^1(k,H) \hookrightarrow H^2(k, \mu_2)$ of pointed sets. We now give another description of $\gamma$ and in particular show that it is a group homomorphism; hence the set of $H(k)$-orbits forms a principal homogenous space for $\ker\gamma$. Note that even though both the source and target of $\gamma$ are groups, there is a priori no reason for $\gamma$ to be a group homomorphism. The short
exact sequence 
\begin{equation}\label{eq:GtoH}
1 \rightarrow \mu_2 \rightarrow R_0 \rightarrow R_0/\mu_2 \rightarrow 1
\end{equation}
of finite abelian group schemes over $k$ gives rise to the long exact sequence in cohomology
$$1\rightarrow \langle \pm1 \rangle \rightarrow R_0(k) \rightarrow R_0/\mu_2(k) \rightarrow k^\times/k^{\times2} \rightarrow H^1(k, R_0) \rightarrow H^1(k,R_0/\mu_2) \xrightarrow{\delta} H^2(k, \mu_2) \!=\! \mathrm{Br}(k)[2].$$
By the definition of the connecting homomorphism, we see that $\delta=\gamma.$ Let $H^1(k, R_0/\mu_2)_{ker}~:=~\ker\delta$ denote the kernel.
Then we have the following short exact sequence
\begin{equation}\label{eq:GtoH2}
1 \rightarrow k^\times/k^{\times2}\langle H \rangle \rightarrow H^1(k,R_0) \rightarrow H^1(k,R_0/\mu_2)_{ker} \rightarrow 1,
\end{equation}
where $\langle H \rangle$ denotes the image of $R_0/\mu_2(k)$ in $k^\times/k^{\times2}$. The group $\langle H \rangle$ can be nontrivial only when $n$ is divisible by $4$; in this case $\langle H \rangle$ is a finite elementary abelian $2$-group corresponding to the quadratic field extensions $K$ of $k$ that are contained in the algebra $L$. In that case, a factorization of $g(x)$ into two even degree polynomials conjugate over $K$ gives a rational point of $R_0/\mu_2(k)$ which is not in the image of $R_0(k)$. Recall that $R_0=G_f$ is the stabilizer for the action of the group $G=\SL(W)$ (Theorem \ref{rational orbits}). Therefore, \eqref{eq:GtoH2} describes how $G(k)$-orbits combine into $H(k)$-orbits and reflects the extra relations in Theorem~\ref{thm:orbitH}.

We now give a more concrete description of $H^1(k,R_0/\mu_2)_{ker}$ in terms of the algebras $L$ and $k$. The above short exact sequence maps
surjectively to the short exact sequence
$$1 \rightarrow k^\times/k^{\times2}\langle I \rangle \rightarrow (L^\TIMES/L^{\TIMES2})_{N \equiv 1} \rightarrow (L^\TIMES/L^{\TIMES2}k^\times)_{N \equiv 1} \rightarrow 1,$$
where $\langle I \rangle$ is the finite elementary abelian subgroup corresponding to all of the quadratic extensions $K$ of $k$ that are contained in $L$. We have $\langle I \rangle = \langle H \rangle$ except in the case when $n$ is not divisible by $4$ and there is a (unique) quadratic extension field $K$ contained in $L$, in which case, the kernel of the map from $H^1(k,R_0)$ to $(L^\TIMES/L^{\TIMES2})_{N \equiv 1}$ has order $2$ whereas the map from $H^1(k,R_0/\mu_2)_{ker}$ to  $(L^\TIMES/L^{\TIMES2}k^\times)_{N \equiv 1}$ is a bijection. In all other cases, these maps have isomorphic kernels (of order $1$ or $2$).

The existence and surjectivity of the map from $H^1(k,R_0/\mu_2)_{ker}$ to $(L^\TIMES/L^{\TIMES2}k^\times)_{N \equiv 1}$ in the above paragraph follows formally from exactness. More canonically, $(L^\TIMES/L^{\TIMES2}k^\times)_{N \equiv 1}$ can be viewed as the subgroup of $H^1(k,R/\mu_2)$ consisting of elements that map to $0$ in $H^2(k,\mu_2)$ under the connecting homomorphism in Galois cohomology and to $0$ in $H^1(k,\mu_2)$ under the map induced by $N\!:\!R/\mu_2\rightarrow\mu_2$. The natural map $H^1(k,R_0/\mu_2)\rightarrow H^1(k,R/\mu_2)$ sends $H^1(k,R_0/\mu_2)_{ker}$ to this subgroup. The kernel of this map is generated by a class $W_H\in H^1(k,R_0/\mu_2).$ The points of the principal homogeneous space $W_H$ over an extension field $E$ are the odd factorizations of $g(x)$ that are rational over $E$.

Since the finite group scheme $H_{A,B} = R_0/\mu_2$ is self-dual, we obtain a cup product pairing
$$H^1(k, R_0/\mu_2) \times H^1(k, R_0/\mu_2) \rightarrow H^2(k,\mu_2).$$
The connecting homomorphism $\delta : H^1(k, R_0/\mu_2) \rightarrow H^2(k,\mu_2)$ (and hence also $\gamma$) is given by the cup product against the class $W_H$ in $H^1(k, R_0/\mu_2)$ (\cite[Proposition 10.3]{PSh}).


Theorems \ref{rational orbits} and \ref{thm:orbitH} have a number of applications to the arithmetic of hyperelliptic curves, which we study in a forthcoming paper \cite{BGW}. A binary form $f(x,y)$ of degree $n=2g+2$ with nonzero discriminant determines a smooth hyperelliptic curve $C:z^2=f(x,y)$ of genus $g$. Here we view $C$ as embedded in weighted projective space $\P(1,1,g+1).$ Denote the Jacobian of $C$ by $J$. Then $J[2]$ is canonically isomorphic to $R_0/\mu_2.$ Under this isomorphism, the self-duality of $R_0/\mu_2$ is given by the Weil pairing on $J[2].$ 
In~\cite{BGW}, we use this connection to show that a positive proportion of 
hyperelliptic curves over $\Q$ of a fixed genus $g$ have points locally at every place of $\Q$ but have no points over any odd degree extension of $\Q$.

\subsection{The two obstructions}

In the previous section, we assumed that rational orbits with invariant $f$ exist and studied the set of $H(k)$-orbits on $V_f(k)$. By Theorem \ref{thm:liftsuf}, we know that there are two obstructions to existence of a $H(k)$-orbit with invariant $f$: the nonvanishing of a class $d_f\in H^2(k,H_f)$, and the possibility that we are not working with the correct pure inner form. In this subsection, we compute the class $d_f\in H^2(k,H_f)$ and when $d_f=0$, and describe the set of pure inner forms $H^c$ of $H$ for which $V^c_f(k)$ is nonempty.


\begin{theorem}\label{thm:dfG} Let $f(x,y)=f_0x^n+\cdots+f_ny^n$ be a binary form of even degree $n$ such that $\Delta(f)$ and $f_0$ are both nonzero. Write $f(x,1)=f_0g(x)$ for some monic polynomial $g(x)$ and let $L=k(x)/(g)=k[\beta]$ be the associated \'{e}tale algebra of rank $n$ over $k$. The groups $G=\SL_n$ and $H=\SL_n/\mu_2$ act on $V=\Sym^2(n)\oplus\Sym^2(n)$. The stabilizer $G_f$ {\rm (}resp.\ $H_f${\em )} associated to an element of $V_f(k^s)$ is the finite group scheme $(\Rlk\mu_2)_{N=1}$ \emph{(}resp.\ $(\Rlk\mu_2)_{N=1}/\mu_2$\emph{)}. Let $\delta_0$ denote the connecting homomorphism $H^1(k,\mu_2)\rightarrow H^2(k,G_f)$ appearing in \eqref{eq:dfone}. Let $d_f^G\in H^2(k,G_f)$ \emph{(}resp.\ $d_f^H\in H^2(k,H_f)$\emph{)} denote the obstruction class for the existence of $G(k)$- \emph{(}resp.\ $H(k)$-\emph{)} orbits with invariant $f$ as defined in $\S\ref{sec:cohomobst}$. Then $d_f^G$ is the image of $f_0$ under $\delta_0,$ and the natural map $H^2(k,G_f)\rightarrow H^2(k,H_f)$ sends $d_f^G$ to $d_f^H$.
\end{theorem}

\begin{proof}
The statement regarding the stabilizer schemes $G_f$ and $H_f$ has been proved in Theorems~\ref{rational orbits} and \ref{thm:orbitH}, respectively. We now compute $d_f^G$ following its definition given in~\S\ref{sec:cohomobst}. Let $A_0$ denote the matrix with 1's on the anti-diagonal and 0's elsewhere. Let $h(x)\in k^s[x]$ be a polynomial such that $N_{L^s/k^s}(h(\beta))=f_0.$ Let~$T$ be a $k$-rational linear operator on $W$ that is self-adjoint with respect to $A_0$ and has characteristic polynomial $g(x)$ (\cite[\S 2.2]{SW}). Then the element $v=(A_0h(T),A_0Th(T))\in V_f(k^s)$ has invariant $f$. We need to pick $g_\sigma\in G(k^s)$ such that $g_\sigma\gl{v}{\sigma} = v$ for every $\sigma\in\Gal(k^s/k)$. We take $g_\sigma$ to be of the form $g_\sigma=j_\sigma(T)$ for some polynomial $j_\sigma(x)\in k^s[x]$ such that $j_\sigma(\beta)^2 = (\gl{h}{\sigma})(\beta)/h(\beta).$  By writing $(\gl{h}{\sigma})(\beta)$, we wish to emphasize that $\sigma$ is not acting on $\beta$, and hence for any polynomial $h'(x)\in k^s[x]$, we have
$$\gl{(N_{L^s/k^s}(h'(\beta)))}{\sigma} = N_{L^s/k^s}((\gl{h'}{\sigma})(\beta)).$$
Let $\sqrt{h}(x)\in k^s[x]$ denote a polynomial such that $(\sqrt{h}(\beta))^2 = h(\beta).$  Set $j_\sigma(x)\in k^s[x]$ to be the polynomial such that $j_\sigma(\beta)=(\gl{\sqrt{h}}{\sigma})(\beta)/\sqrt{h}(\beta).$ By definition, $d_f^G$ is then the 2-cocycle
$$(\sigma,\tau)\mapsto j_\sigma(\beta)\gl{(j_\tau(\beta))}{\sigma}j_{\sigma\tau}(\beta)^{-1}.$$
On the other hand, let $\sqrt{f_0}$ denote the square root of $f_0$ such that $\sqrt{f_0}=N_{L^s/k^s}(\sqrt{h}(\beta)).$ Then the 1-cocycle $\sigma\mapsto \gl{\sqrt{f_0}}{\sigma}/\sqrt{f_0}$ corresponds to the class of $f_0\in k^\times/k^{\times2}.$ To compute $\delta_0(f_0),$ for each $\sigma\in\Gal(k^s/k)$ we need to find an element in $L^s$ whose norm to $k^s$ is $\gl{\sqrt{f_0}}{\sigma}/\sqrt{f_0}$. A natural choice is $j_\sigma(\beta).$ The equality $d_f^G=\delta_0(f_0)$ is then clear. The second statement is also clear from the above computation for $d_f^G$.
\end{proof}

\medskip

Since $G$ has no nontrivial pure inner forms, the vanishing of $d_f^G$ is sufficient for the existence of rational orbits. For $H$, there is a second (Brauer-type) obstruction coming from the pure inner forms of $H$. 

\begin{theorem}\label{thm:innerforms} 
Let $f(x,y)=f_0x^n+\cdots+f_ny^n$ be a binary form of even degree $n$ such that $\Delta(f)$ is nonzero. Let $d_f^G\in H^2(k,G_f)$ \emph{(}resp.\ $d_f^H\in H^2(k,H_f)$\emph{)} denote the obstruction class for the existence of $G(k)$- \emph{(}resp.\ $H(k)$-\emph{)} orbits with invariant $f$. Consider the following diagram:
\begin{displaymath}
\xymatrix{
&H^1(k,H)\ar[d]^{\delta_2}&\\
H^1(k,H_f)\ar[r]^{\delta}&H^2(k,\mu_2)\ar[r]^{\alpha}&H^2(k,G_f),}
\end{displaymath}
where $\delta,\delta_2$ are the connecting homomorphisms in Galois cohomology and $\alpha$ is induced by the diagonal inclusion $\mu_2\rightarrow G_f.$ Suppose $d_f^H = 0.$ Then $d_f^G$ is the image of some $d\in H^2(k,\mu_2)$, where $d$ lies in the image of $\delta_2.$ The pure inner forms of $H$ for which rational orbits exist with invariant $f$ correspond to classes $c\in H^1(k,H)$ such that $\alpha\delta_2(c)=d_f^G$ in $H^2(k,G_f)$. 
\end{theorem}

\begin{proof}
Fix any $v\in V_f(k^s).$ Choose $g_\sigma\in H(k^s)$ for each $\sigma\in \Gal(k^s/k)$ such that $g_\sigma\gl{v}{\sigma} = v.$ Since $d_f^H=0,$  by Lemma \ref{lem:prolemma} we may pick $g_\sigma$ such that $c=(\sigma\rightarrow g_\sigma)$ is a 1-cocycle in $H^1(k,H)$. Lift each $g_\sigma$ arbitrarily to $\w{g}_\sigma\in G(k^s).$ Since the center of $G(k^s)$ acts trivially on $V$, we have $\w{g}_\sigma \gl{v}{\sigma} = v$ for every $\sigma\in \Gal(k^s/k)$. The 2-cocycle $d_f^G$ in $H^2(k,\mu_2)$ is then given by 
\begin{equation}\label{delta2c}
(d_f^G)_{\sigma,\tau} = \w{g}_\sigma \gl{\w{g}_\tau}{\sigma}\w{g}_{\sigma\tau}^{-1},\end{equation}
which is exactly the image of $c$ under $\delta_2$.

For the second statement, choose $g\in \GL(V)(k^s)$ such that $g_\sigma=g^{-1}\gl{g}{\sigma}$ for every $\sigma\in \Gal(k^s/k)$. From the definition of $g_\sigma$, we see that $gv\in V^c_f(k).$ For every $v'\in V_f(k^s)$, let $\iota_{v'}:H_f(k^s)\rightarrow H_{v'}$ denote the canonical isomorphism. Then we have a Galois invariant isomorphism $H_f(k^s)\rightarrow H^c_{gv}(k^s)$ sending $b\in H_f(k^s)$ to $\iota_v(b).$ Let $\iota$ denote the following composition:
$$\iota: H^1(k,H_f)\xrightarrow{\sim}H^1(k,H^c_{gv})\rightarrow H^1(k, H^c)\xrightarrow{\sim} H^1(k,H),$$
where the last map is the bijection given by $(\sigma\rightarrow d_\sigma)\mapsto(\sigma\rightarrow d_\sigma g_\sigma).$

\begin{lemma}\label{lem:H1Hf} For any $b\in H^1(k, H_f)$, we have $$\delta_2(\iota(b)) = \delta(b) + \delta_2(c).$$
\end{lemma}
This lemma follows from a direct computation similar to the proof of Lemma \ref{lem:prolemma}.

Proposition \ref{prop:twistorbit} states that the set of pure inner forms of $H$ for which rational orbits exist with invariant $f$ is in bijection with the image of $H^1(k,H_f)$ under $\iota$. Since $\delta_2$ is injective, Lemma \ref{lem:H1Hf} implies that this set is in bijection with $\delta(H^1(k,H_f))+\delta_2(c)$, which equals $\alpha^{-1}(\alpha\delta_2(c))$ by exactness. Theorem~\ref{thm:innerforms} now follows since,  by (\ref{delta2c}), we have $d_f^G=\alpha\delta_2(c)$.
\end{proof}

\section{Integral orbits}\label{intorbits}

In this section, we discuss the orbits of the group $G(\Z)=\SL_n(\Z)$ on the free $\Z$-module $V(\Z)=
\Sym_2\Z^n\oplus\Sym_2\Z^n$ 
of symmetric matrices $(A,B)$ having entries in $\Z$. Even though Galois cohomology was very useful in the previous sections to study rational orbits, we will see in this section that one will generally need different techniques to study integral orbits. 

Associated to an integral orbit we have the invariant binary $n$-ic form
$f(x,y) = \disc(xA - yB) = f_0x^{n} + \cdots + f_ny^n$ with integral coefficients.
We assume as above that the integers $\Delta(f)$ and $f_0$ are both nonzero, and write
$f(x,1) = f_0g(x)$. The polynomial $g(x)$ is separable over $\Q$, but its coefficients will not necessarily be integers (when $f_0 \neq \pm1$). In this case, the image $\theta$ of $x$ in the \'etale algebra $L = \Q[x]/g(x)$ will not necessarily be an algebraic integer.

The rational orbits with this binary form $f$ correspond to equivalence classes of pairs $(\gamma, t)$. Here $\gamma$ is an invertible element in the \'etale algebra $L$ and $t$ is an invertible element of $\mathbb Q$ satisfying $t^2 = f_0N(\gamma)$.
The equivalence relation is $(\gamma, t) \sim (c^2 \gamma, N(c)t)$ for all $c \in L^\times$. In this section, we specify the additional data that
determines an integral orbit in this rational orbit.

Recall that an {\it order} $R$ in $L$ is a subring that is free of rank $n$ over $\Z$ and generates $L$ over $\Q$.  The ring $\Z[\theta]$ generated by $\theta$ will not be an order in $L$ when the coefficients of $g(x)$ are not integers. However, there is a natural order $R_f$ contained in $L$ which is determined by the integral binary form $f(x,y)$.
This order $R_f$ as a $\Z$-module was first introduced by Birch and Merriman~\cite{BM} and proved to be an order by Nakagawa \cite{N}.  A basis-free description was discovered by Wood~\cite{Wood}, namely, $R_f$  is the ring of the global sections of the structure sheaf of the subscheme $S_f$ of $\mathbb P^1$ defined by the homogeneous equation $f(x,y) = 0$ of degree $n$.  The ring $R_f$ possesses a natural $\Z$-basis, namely,  
$$R_f = \Span_\Z\{ 1,\zeta_1,\zeta_2,\ldots,\zeta_{n-1}\},$$
where
\begin{equation}
\zeta_k=f_0\theta^k+f_1\theta^{k-1}+\cdots+f_{k-1}\theta.
\end{equation}
Note that the $\zeta_k$ are all algebraic integers, even though $\theta$ might not be.  One easily checks (\cite{BM}) the remarkable equality $\disc(f)=\disc(R_f)$.

A {\it fractional ideal} $I$ for an order $R$ is a free abelian subgroup of rank $n$ in $L$, which is stable under multiplication by $R$.  The norm $N(I)$ is defined to be the positive rational number that is the quotient of the index of $I$ in $M$ by the index of $R$ in $M$, where $M$ is any lattice in $L$ that contains both $I$ and $R$. If the fractional ideal $I$ is contained in $R$, so defines an ideal of $R$ in the usual sense, then $N(I)$ is its index in $R$. An {\it oriented fractional ideal} for an order $R$ is a pair $(I,\varepsilon)$, where $I$ is any fractional ideal of $R$ and $\varepsilon = \pm1$ gives the {\it orientation} of $I$.  The norm of an oriented ideal $(I,\varepsilon)$ is defined to be the nonzero rational number $\varepsilon N(I)$. For an element $\kappa\in L^\times$, the product $\kappa (I,\varepsilon)$ is defined to be the oriented fractional ideal $(\kappa I,{\rm sgn}(N(\kappa))\varepsilon).$  Then $N(\kappa  (I,\varepsilon))=N(\kappa) N(I, \varepsilon)$ in $\Q^\times$. In practice, we denote an oriented ideal $(I, \varepsilon)$ simply by $I$, with the orientation $\varepsilon = \varepsilon(I)$ on $I$ being understood.

We say that a fractional ideal $I$ is {\it based} if it comes with a fixed ordered basis over $\Z$. If the order $R$ and the fractional ideal $I$ are both based, then we can define the orientation of $I$ as the sign of the determinant of the $\Z$-linear transformation taking the chosen basis of $I$ to the basis of $R$. The norm of this oriented fractional ideal is then equal to the actual determinant.  Changing the basis by an element of $\SL_n(\Z)$ does not change the orientation $\varepsilon$ of $I$ or the norm $N(I)$ in $\Q^\times$.

The binary form $f(x,y)$ not only defines an order $R_f$ in $L$, but also a collection of based fractional ideals $I_f(k)$ for $k = 0,1,2 \ldots, n-1$ (see~\cite{Wood}). The ideal $I_f(0) = R_f$ and for $k > 0$ the ideal $I_f(k)$ has a $\Z$-basis $\{1,\theta,\theta^2,\ldots,\theta^k,\zeta_{k+1},\ldots,\zeta_{n-1}\}.$
This gives $I_f(k)$ an orientation relative to $R_f$, and the norm of the oriented ideal $I_f(k)$ is equal to $1/f_0^k$. We have inclusions $R_f \subset I_f(1) \subset I_f(2) \subset \cdots \subset I_f(n-1)$.

Let $I_f = I_f(1)$. Then
we find by explicit computation that $I_f(k) = I_f^k$.
As shown by Wood~\cite{Wood}, abstractly the
fractional ideal $I_f$ is the module of global sections of the pullback of the line bundle $\mathcal O(1)$ on $\mathbb P^1$ to the subscheme $S_f$ defined by the equation $f(x,y) = 0$, and the ideals $I_f(k)$ are the global sections of the pullbacks of the line bundles $\mathcal O(k)$ . We say that the form $f(x,y)$ is {\it primitive} if the greatest common divisor of its coefficients is equal to $1$. When $f(x,y)$ is primitive, the scheme $S_f = \Spec(R_f)$ has no vertical components and is affine. In this case, the pullbacks of these line bundles have no higher cohomology, and the ideals $I_f(k) = I_f^k$ are all projective $R_f$-modules.

The oriented fractional ideal $I_f(n-1)$ has a power basis $\{1,\theta, \theta^2, \ldots, \theta^{n-1}\}$. When the form $f(x,y)$ is primitive, this fractional ideal is a projective, hence a proper, $R_f$-module. In this case, the ring $R_f$ has a simple definition as the endomorphism ring of the lattice $ \Span_\Z\{ 1,\theta, \theta^2,\ldots, \theta^{n-1}\}$ in the algebra $L$.

There is also a nice interpretation of the oriented fractional ideal
$I_f(n-2) = I_f^{n-2}$
in terms of the trace pairing on $L$. Define a nondegenerate bilinear pairing $\langle\:\,,\: \rangle_f: R_f \times I_f^{n-2} \rightarrow \Z$ by taking
$\langle \lambda, \mu \rangle_f$ as the coeffiecient of $\zeta_{n-1}$ in the product $\lambda\mu$.
Define $f'(\theta)$ in $L^\times$ by the formula $f'(\theta) = f_0g'(\theta)$. Then $f'(\theta)$ lies in the fractional ideal $I_f^{n-2}$. A computation due to Euler {\bf (reference?)} shows that the above bilinear pairing is given by the formula
$$\langle \lambda, \mu \rangle_f = \Tr(\lambda\mu/f'(\theta)),$$
where the trace is taken from $L$ to $\Q$. We have an inclusion $R_f \subset (1/f'(\theta))I_f^{n-2}$ and the index is the absolute value of $\Delta(f)$. In fact, the oriented
fractional ideal $(1/f'(\theta))I_f^{n-2}$ has norm $1/\Delta(f)$. This is precisely the ``inverse different''---the dual module to $R_f$ in $L$ under the trace pairing. When $f(x,y)$ is primitive, the dual module is projective and the ring $R_f$ is Gorenstein.

The oriented fractional ideal
$I_f(n-3) = I_f^{n-3}$
appears in the study of integral orbits. Before introducing the action of $\SL_n(\Z)$, we first describe the elements in $V(\Z)$ using a general theorem of
Wood (see~\cite[Theorems~4.1 \& 5.7]{Wood1},
or \cite[Theorem~16]{hclI} and \cite[Theorem~4]{hclII}
for the special cases $n=2$ and $n=3$):
\begin{theorem}[Wood]\label{woodsym}
The
elements of 
$\Sym_2(\Z^n)\oplus\Sym_2(\Z^n)$
having a given invariant binary $n$-ic form $f$
with nonzero discriminant $\Delta$ and nonzero first coefficient $f_0$
are in bijection with the equivalence classes of pairs
$(I,\alpha)$, where $I\subset L$ is a based fractional ideal of $R_f$, $\alpha\in
L^\times$, $I^2\subseteq \alpha I_f^{n-3}$ as fractional ideals, and
$N(I)^2=N(\alpha)N(I_f^{n-3}) = N(\alpha)/f_0^{n-3} \in \Q^\times$.
Two pairs
$(I,\alpha)$ and $(I^*,\alpha^*)$ are equivalent if
there exists $\kappa\in L^\times$ such that $I^*=\kappa I$ and $\alpha^*=\kappa^2\alpha.$
\end{theorem}

The way to recover a pair $(A,B)$ of symmetric $n\times n$ matrices from a
pair $(I,\alpha)$ above is by taking the coefficients of $\zeta_{n-1}$
and $\zeta_{n-2}$ in the image of the map
\begin{equation}\label{mt1}
\frac1\alpha\times:I\times I\to I_f^{n-3}
\end{equation}
in terms of the $\Z$-basis of $I$.

Next, note that the group $G(\Z)=\SL_n(\Z)$ acts naturally on $V(\Z)=\Sym_2(\Z^n)\oplus\Sym_2(\Z^n)$.
It also acts on the bases of the based fractional ideals $I$ in the corresponding pairs $(I,\alpha)$, and preserves the norm and orientation.  Thus, when considering $\SL_n(\Z)$-orbits, we may drop the bases of $I$ and view $I$ simply as an oriented fractional ideal ideal.  We thus obtain:

\begin{corollary}\label{woodsym2}
The orbits of $\SL_n(\Z)$ on
$
\Sym_2(\Z^n)\oplus\Sym_2(\Z^n)
$ 
having a given invariant binary $n$-ic form $f$ with nonzero discriminant $\Delta$ and nonzero first coefficient $f_0$
are in bijection with equivalence classes of pairs
$(I,\alpha)$, where $I\subset L$ is an oriented fractional ideal of $R_f$, $\alpha\in
L^\times$, $I^2\subseteq \alpha I_f^{n-3}$, and
$N(I)^2=N(\alpha)N(I_f^{n-3}) = N(\alpha)/f_0^{n-3}$.
Two pairs
$(I,\alpha)$ and $(I^*,\alpha^*)$ are {\em equivalent} if
there exists $\kappa\in L^\times$ such that $I^*=\kappa I$ and $\alpha^*=\kappa^2\alpha.$
The stabilizer in $\SL_n(\Z)$ of a nondegenerate element in $\Sym_2(\Z^n)\oplus\Sym_2(\Z^n)$
 having invariant binary form $f$ is the finite elementary abelian $2$-group $S^\times[2]_{N = 1}$, where $S$ is the endomorphism ring of $I$ in $L$.
\end{corollary}

We can specialize this result to the case when the order $R_f$ is maximal in $L$ (which occurs, for example, when the discriminant $\Delta(f)$ is squarefree). Then the set of oriented fractional ideals of $R_f$ form an abelian group under multiplication, and the principal oriented ideals form a subgroup.
The {\it oriented class group} $C^*$ is then defined as the quotient of the group of invertible oriented ideals by the subgroup of principal oriented ideals. The elements of this group are called the invertible oriented ideal classes of $R_f$, and two oriented ideals $(I, \varepsilon)$ and $(I', \varepsilon')$ of $R_f$ are in the same oriented ideal class if $(I', \varepsilon') = \kappa \cdot (I,\varepsilon)$ for some $\kappa \in L^\times$.   Note that the oriented class group is isomorphic to the usual class group of $R_f$ if there is an element of $R_f^\times$ with norm $-1$; otherwise, it is an extension of the usual class group by $\Z/2\Z$, where the generator of this $\Z/2\Z$ is given by
the oriented ideal $(R_f,-1)$ of $R_f$.  (In the case of a binary form with positive discriminant, when $R_f$ is an order in a real number field, the oriented class group coincides with what is usually called the {\it narrow class group}).

When $R_f$ is maximal, integral orbits $(A,B)$ with invariant $f$ will exist if and only if the class of the oriented ideal $I_f(n-3) = I_f^{n-3}$ is a square in the oriented ideal class group (this will certainly hold when $n$ is odd). If the class of $I_f^{n-3}$ is a square, we can find a pair $(I,\alpha)$ satisfying $I^2 = \alpha I_f^{n-3}$ and $N(I)^2 = N(\alpha)/f_0^{n-3}$. In this case, the set of orbits is finite and forms a principal homogeneous space for an elementary abelian $2$-group that is an extension of the group of elements of order $2$ in the oriented class group by the group $(R_f^\times/R_f^{\times2})_{N = 1}$.
The number of distinct integral orbits with binary form $f(x,y)$ is given by the formula
$$2^{r_1 + r_2 - 1}\#C^*[2]$$
where $r_1$ and $r_2$ are the number of real and complex places of $L$ respectively and $C^*[2]$ is the subgroup of elements of order $2$ in the oriented class group $C^*$.

\medskip
We end with a comparison of the integral and rational orbits with a fixed invariant form $f$ for the action of $G=\SL_n$ on $V=\Sym_2(n)\oplus \Sym_2(n).$ Let $f(x,y) = f_0x^n + f_1x^{n-1}y+\cdots+f_ny^n$ be an integral binary form of degree~$n$ with $\Delta(f) \neq 0$ and $f_0 \neq 0$. Write $f(x,1) = f_0 g(x)$ with $g(x) \in \Q[x]$ and let $L = \Q[x]/(g(x))$. Recall from \S\ref{spinrep} that the orbits $v = (A,B)$ of $\SL_n(\Q)$ on $V(\Q)$ with invariant~$f$ correspond bijectively to the equivalence classes of pairs $(\gamma, t)$, with $\gamma \in L^\times$ and $t \in \Q^\times$ satisfying $t^2 = f_0N(\gamma)$. More precisely, the $\SL_n(\Q)$-orbit of the bilinear form $A$ is given by the pairing
$$\langle \lambda, \mu \rangle_{\gamma} = \Tr(\lambda\mu/\gamma g'(\theta)).$$
using the oriented basis $t(1 \wedge \theta \wedge \theta^2 \wedge \ldots \wedge \theta^{n-1})$ of $\wedge^nL$. It follows that $\langle \lambda, \mu \rangle_A$ is equal to the coefficient of $\theta^{n-1}$ in the expansion of the product $\lambda\mu/\gamma$ using this oriented basis.

On the other hand, an integral orbit $(A,B)$ is given by the equivalence class of the pair $(I,\alpha)$ with $I^2 \subset \alpha I_f^{n-3}$ and $N(I)^2 = N(\alpha)/f_0^{n-3}$. For $\lambda$ and $\mu$ in the oriented fractional ideal $I$, the bilinear form
$\langle \lambda, \mu \rangle_A$ is equal to the coefficient of $\zeta_{n-1}$ in the expansion of the product $\lambda\mu/\alpha$ with respect to the natural basis of $I_f(n-3)$. Since $\zeta_{n-1} = f_0 \theta^{n-1} + f_1 \theta^{n-2} + \cdots + f_{n-2} \theta$ in $L$, we see that the corresponding rational orbit has parameters
$$ \gamma = f_0\alpha,$$
$$ t = f_0^{n-1}N(I).$$
Similarly, the bilinear form $\langle \lambda, \mu \rangle_B$ is equal to the coefficient of $\zeta_{n-2}$ in the expansion of the product $\lambda\mu/\alpha$ with respect to the natural basis of $I_f(n-3)$. Note that we obtain Gram matrices for these two bilinear forms by using the basis of the ideal $I$ that maps to the basis element
$$N(I)(1 \wedge \zeta_1 \wedge \zeta_2 \wedge \ldots \wedge \zeta_{n-1}) = N(I)f_0^{n-1}( 1 \wedge \theta \wedge \theta^2 \wedge \ldots \theta^{n-1}) = t( 1 \wedge \theta \wedge \theta^2 \wedge \ldots \theta^{n-1}) $$
of the top exterior power of $I$ over $\Z$.

If we fix a rational orbit with integral form $f(x,y)$, then the parameters $(\gamma, t)$ determine both $\alpha$ and $N(I)$ by the above formulae. The rational orbit has an integral representative if and only if one can find an oriented fractional ideal $I$ for $R_f$ satisfying $I^2\subseteq \alpha I_f^{n-3}$ and $N(I)=N(\alpha)N(I_f^{n-3}) = N(\alpha)/f_0^{n-3}$.
The distinct integral orbits in this rational orbit correspond to the different possible choices for the oriented fractional ideal $I$ satisfying these two conditions. We note that there is at most one choice when the order $R_f$ is maximal in $L$. In that case, the fractional ideal $I$ is determined by the identity $I^2 = \alpha I_f^{n-3}$, and its orientation by the identity $N(I) = N(\alpha)/f_0^{n-3}$.

When $n$ is odd, there is a canonical integral orbit with invariant binary $n$-ic form $f(x,y)$. This has parameters $(I, \alpha) = (I_f^{(n-3)/2}, 1)$. The corresponding rational orbit has parameters $(\gamma, t) = (f_0, f_0^{(n+1)/2})$.

\end{document}